\documentclass[12pt]{iopart}

\usepackage{fixltx2e} 

\usepackage{cmap} 

\usepackage[T1]{fontenc}
\usepackage[utf8]{inputenc}
\usepackage{graphicx}
\usepackage{placeins}

\usepackage{verbatim}

\usepackage{setspace}

\usepackage{lmodern} 
\usepackage[scale=0.88]{tgheros} 

\usepackage{bm} 

\usepackage{bbold}

\expandafter\let\csname equation*\endcsname\relax
\expandafter\let\csname endequation*\endcsname\relax
\usepackage{amsmath,amsbsy,amsgen,amscd,amsthm,amsfonts,amssymb} 

\usepackage[centering,top=1.5in,bottom=1.2in,left=1.4in,right=1.4in]{geometry}

\usepackage{titling}
\usepackage{booktabs,longtable,tabu} 
\usepackage[font=small,margin=12pt,labelfont={sf,bf},labelsep={space}]{caption}

\usepackage{enumitem}
\setitemize{itemsep=0pt} 
\setenumerate{itemsep=0pt}
\setlist[1]{labelindent=\parindent,
 }

\usepackage[usenames,dvipsnames]{xcolor}
\definecolor{dark-gray}{gray}{0.3}
\definecolor{dkgray}{rgb}{.4,.4,.4}
\definecolor{dkblue}{rgb}{0,0,.5}
\definecolor{medblue}{rgb}{0,0,.75}
\definecolor{rust}{rgb}{0.5,0.1,0.1}

\usepackage{url}
\usepackage[colorlinks=true]{hyperref}
\hypersetup{linkcolor=dkblue}    
\hypersetup{citecolor=rust}      
\hypersetup{urlcolor=rust}     

\usepackage[final]{microtype} 

\newtheoremstyle{myThm} 
    {\topsep}                    
    {\topsep}                    
    {\itshape}                   
    {}                           
    {\bfseries}                   
    {.}                          
    {.5em}                       
    {}  

\newtheoremstyle{myRem} 
    {\topsep}                    
    {\topsep}                    
    {}                   
    {}                           
    {}                   
    {.}                          
    {.5em}                       
    {}  

\newtheoremstyle{myDef} 
    {\topsep}                    
    {\topsep}                    
    {}                   
    {}                           
    {\bfseries}                   
    {.}                          
    {.5em}                       
    {}  

\theoremstyle{myThm}
\newtheorem{theorem}{Theorem}[section]
\newtheorem{lemma}[theorem]{Lemma}

\newtheorem{corollary}[theorem]{Corollary}

\theoremstyle{myRem}
\newtheorem{remark}[theorem]{Remark}

\theoremstyle{myDef}

\usepackage{fancyhdr}
\usepackage{nopageno} 
\let\originalleft\left
\let\originalright\right
\renewcommand{\left}{\mathopen{}\mathclose\bgroup\originalleft}
\renewcommand{\right}{\aftergroup\egroup\originalright}

\usepackage{mathtools}
\mathtoolsset{centercolon}  

\renewcommand{\phi}{\varphi}

\providecommand{\mathbbm}{\mathbb} 

\newcommand{\R}{\mathbbm{R}}

\newcommand{\cH}{\mathcal{H}}
\newcommand{\cA}{\mathcal{A}}

\usepackage{algpseudocode}
\usepackage[]{algorithm}
\usepackage{graphicx}
\newcommand\NoDo{\renewcommand\algorithmicdo{}}
\newcommand\NoThen{\renewcommand\algorithmicthen{}}

\usepackage{subfigure}

\newcommand{\upperRomannumeral}[1]{\uppercase\expandafter{\romannumeral#1}}

\renewcommand{\hat}{\widehat}

\begin{document}
\title{Ensemble Kalman Methods With Constraints}
\author{David J. Albers$^{1,2}$, Paul-Adrien Blancquart$^3$, Matthew E. Levine$^4$, Elnaz Esmaeilzadeh Seylabi$^5$, Andrew Stuart$^4$}

\address{$^1$ Department of Biomedical Informatics, Columbia University, New York, NY 10032}
\address{$^2$ Department of Pediatrics, Division of Informatics, University of Colorado Medicine, Aurora, CO 80045}
\address{$^3$ Mines ParisTech, PSL Research University, Paris, France}
\address{$^4$ Department of Computational and Mathematical Sciences, California Institute of Technology, Pasadena, CA 91125}
\address{$^5$ Department of Mechanical and Civil Engineering, California Institute of Technology, Pasadena, CA 91125}
\eads{\mailto{david.albers@ucdenver.edu}, \mailto{ paul-adrien.blancquart@mines-paristech.fr}, \mailto{mlevine@caltech.edu}, \mailto{elnaz@caltech.edu}, \mailto{astuart@caltech.edu}}
\begin{abstract}
Ensemble Kalman methods constitute an increasingly important tool in both
state and parameter estimation problems. Their popularity stems from
the derivative-free nature of the methodology which may be readily applied
when computer code is available for the underlying state-space dynamics
(for state estimation) or for the parameter-to-observable map (for
parameter estimation). There are many applications in which it is
desirable to enforce prior information in the form of equality or
inequality constraints on the state or parameter. This paper establishes
a general framework for doing so, describing a widely applicable methodology,
a theory which justifies the methodology, and a set of numerical experiments
exemplifying it.
\end{abstract}

\submitto{\IP}

\section{Introduction}

\subsection{Overview}

Kalman filter based methods have been enormously successful in both
state and parameter estimation problems.
However, a major disadvantage of such methods
is that they do not naturally take constraints into account.
The ability to constrain a system often has a number of advantages that
can play an important role in state and parameter estimation:
they can be used to enforce physicality of modeled systems
(non-negativity of physical quantities, for example); relatedly they can be used
to ensure that computational models are employed only within state
and parameter regimes where the model is well-posed;
and finally the application of constraints may provide robustness to
outlier data.
Resulting improvements in algorithmic efficiency and performance, by
means of enforcing constraints, has been demonstrated in the recent
literature in a diverse set of fields, including process control
\cite{teixeira2010unscented},
biomechanics \cite{bonnet2017constrained}, cell energy metabolism \cite{goffaux2011cell}, medical imaging \cite{lei2012dynamic},
engine health estimation \cite{simon2010constrained}, weather forecasting \cite{janjic2014conservation}, chemical engineering \cite{yang2014inequality}, and hydrology \cite{wang2009state}.
{Within the Kalman filtering literature the need to incorporate constraints is
widely recognized and has been addressed in a systematic fashion by viewing
Kalman filtering from the perspective of optimization. Indeed this optimization
perspective leads naturally to
many extensions, and to the incorporation of constraints in particular.
Including constraints in Kalman filtering, via optimization, lends itself
to an elegant mathematical framework, to a practical computational framework,
and has potential in numerous applications. Surveys of the work may be found in the
papers of Aravkin, Burke and co-workers \cite{aravkin2014optimization,aravkin2017generalized}
and our work in this paper may be viewed as generalizing their perspective to the ensemble setting.}

In the probabilistic view of filtering methods, constraints may be
introduced by moving beyond the Gaussian assumptions that underpin
Kalman methods and imposing constraints through the prior distributions
on states and/or parameters. This, however, can create significant
computational burden as the resulting distributions cannot be represented
in closed form, through a finite number of parameters, in the way that
Gaussian distributions can be.
Here we circumvent this issue by taking the viewpoint that
ensemble Kalman methods constitute a form of derivative-free optimization
methodology, eschewing the probabilistic interpretation.
The ensemble is used to calculate surrogates for derivatives.
With this optimization perspective, constraints may be included
in a natural way. Standard ensemble Kalman methods employ
a quadratic optimization problem encapsulating relative strengths of
belief in the predictions of the model and the data; these optimization
problems have explicit analytic solutions. To impose constraints
the optimization problem is solved only within the constraint set; when
the constraints form a non-empty closed convex set, this constrained
optimization problem has a unique solution.

In this introductory section, we give a literature
review of existing work in this setting, we describe
the contributions in this paper, and we outline notation used
throughout.

\subsection{Literature Review}

Overviews of state estimation using Kalman based methods
may be found in \cite{evensen2009data,reich2015probabilistic,law2015data,carrassi2018data}.
The focus of this article is on ensemble based Kalman methods,
introduced by Evensen in \cite{evensen1994sequential} and further
developed in  \cite{burgers1998analysis,evensen2009data}.
The extension of the ensemble Kalman methodology to
parameter estimation and inverse problems is overviewed in
\cite{oliver2008inverse}, especially for oil reservoir applications,
and in an application-neutral formulation in \cite{iglesias2013ensemble}.
Equipping Kalman-based methods with constraints can be desirable for a
 variety of inter-linked reasons described in the previous subsection: to enforce known physical boundaries
in order to improve estimation accuracy; to operationalize filtering of
a model which is ill-posed in subsets of its state or parameter space;
and to provide robustness to noisy data and outlier events.

In extending the Kalman filter to non-Gaussian settings, a number of
methods may be considered. Particle filters provide the natural methodology
if propagation of probability distributions is required
for state \cite{doucet2001introduction} or parameter \cite{del2006sequential}
estimation. In the optimization setting, there are three primary methodologies:
the extended Kalman filter, the unscented Kalman filter and the
ensemble Kalman filter. The extended Kalman filter is based on linearization of the nonlinear system and therefore needs the computation of derivatives for propagation of the state covariance; this makes them unattractive in
high dimensional problems. Unscented and ensemble Kalman filters, on the other hand, can be considered as particle-based methods which are derivative-free. In the unscented Kalman filter, the particles (sigma points) are chosen deterministically and are propagated through the nonlinear system to approximate the covariance, which is then corrected using the Kalman gain to compute the new sigma points. In the ensemble Kalman filter, the particles (ensemble members) are chosen randomly from the initial ensemble and are propagated through the dynamical system and corrected using the Kalman gain without needing to maintain the covariance.

In \cite{simon2010kalman}, and more recently in \cite{1807.03463},
overviews of different ways to impose constraints in linear and nonlinear
state estimation are presented. To ensure that the estimates satisfy the constraints, moving horizon based estimators that solve a constrained optimization problem have been proposed \cite{robertson1996moving,rao2003constrained}. The paper
\cite{vachhani2005recursive} proposed a recursive nonlinear dynamic data reconciliation (RNDDR) approach based on extended Kalman filtering to ensure that state and parameter estimates satisfy the imposed bounds and constraints. The updated state estimates in this method are obtained by solving an optimization problem instead of using the Kalman gain.
The resulting covariance calculations are, however, still similar to the Kalman
filter: that is, unconstrained propagation and correction involving the Kalman gain, which can affect the accuracy of the estimates. To eliminate this deficiency, \cite{li2018constrained} proposed a Kullback-Leibler based method to update states and error covariances by solving a convex optimization problem involving conic constraints.

On the other hand, the paper \cite{vachhani2006robust} combined the concept of the
unscented transformation \cite{julier2000new} with the RNDDR formulation. In the prediction step, they propose step sizes to scale sigma points asymmetrically to better approximate the covariance information in the presence of lower and upper bounds. Then, for the update of each sigma point, they solve a constrained optimization problem. One disadvantage of this procedure is that the chosen step sizes for scaling the sigma points can only ensure the bound constraints.
The paper \cite{teixeira2010unscented} also tested various algorithms based on constrained optimization, projection \cite{simon2006kalman} and truncation \cite{simon2010constrained} to enforce bound constraints on unscented Kalman filtering.
The paper \cite{mandela2012constrained} developed a class of estimators named constrained unscented recursive estimators to address the limitations of the unscented RNDDR method using optimization-based projection algorithms for obtaining sigma points in the presence of convex, non-convex and bound constraints.

As mentioned earlier, since the corrected covariance is used to compute the sigma points, unscented formulations always require enforcing constraints in both propagation and correction/update steps. In contrast, ensemble-based methods only require constraints to be enforced in the update step.
In this context, the paper \cite{wang2009state} tested projection and accept/reject methods to constrain ensemble members in a post-processing step,
after application of the unconstrained ensemble Kalman filter. In the former, they project the updated ensemble members to the feasible space if they violate the constraints and in the latter they enforce the updated ensemble members to obey the constraints by resampling the dynamic and/or data model errors.
On the other hand, \cite{Prakash2008ConstrainedSE,prakash2010constrained} proposed updating the state estimates in ensemble Kalman filtering by solving a constrained optimization problem while truncating the Gaussian distribution of the initial ensemble.
The paper \cite{janjic2014conservation} demonstrated how to enforce a physics-based conservation law on an
ensemble Kalman filtering based state estimation problem by formulating the filter update as a set of quadratic programming problems arising from a  linear data acquisition model
subject to linear constraints.
Here we develop this body of work on constraining ensemble Kalman
techniques, providing a unifying framework with an underpinning theoretical
basis.

\subsection{Our Contribution}

The preceding literature review demonstrates that the imposition of
constraints on state and parameter estimation procedures is highly
desirable. It also indicates that ensemble Kalman methods offer
the most natural context in which to attempt to do this, as extended
Kalman methods do not scale well to high dimensional state or parameter
space, whilst the unscented filter does not lend itself as naturally
to the incorporation of constraints.

In this paper we build on the application-specific papers
\cite{wang2009state,janjic2014conservation} which
demonstrate how to impose a number of particular constraints on
ensemble based parameter and state estimation problems
respectively. We formulate a very general methodology which
is application-neutral and widely applicable, thereby making the
ideas in \cite{wang2009state,janjic2014conservation} accessible
to a wide community of researchers working in inverse problems
and state estimation. We also
describe a straightforward mathematical analysis which
demonstrates that the resulting algorithms are well-defined
since they involve the solution of quadratic minimization problems
subject to convex constraints at each step of the algorithm; these
optimization problems have a unique solution. And finally we showcase the
methodology on two applications, one from biomedicine and one from
seismology.
All of the algorithms discussed are clearly stated in pseudo-code.

Section \ref{sec:state} outlines the ensemble Kalman (EnKF) methodology
for state estimation, with and without constraints. In section
\ref{sec:inverse} the same program is carried out for ensemble
Kalman inversion (EKI). Section \ref{sec:num} describes the numerical
experiments which illustrate the foregoing ideas.

\subsection{Notation}
Throughout the paper we use $\mathbb{N}$ to denote the positive integers $\{1,2,3, \cdots \}$
and $\mathbb{Z}^+$ to denote the non-negative integers $\mathbb{N} \cup \{0\}=\{0,1,2,3, \cdots \}.$
The matrix $I_M$ denotes the identity on $\mathbb{R}^M.$
We use $|\cdot|$
to denote the Euclidean norm, and the corresponding
inner-product is denoted $\langle \cdot, \cdot \rangle.$
A symmetric, square matrix $A$
is positive definite (resp. positive semi-definite) if the quadratic
form $\langle u, Au \rangle$ is positive (resp. non-negative)
for all $u \ne 0$. By
$| \cdot |_B $ we denote the weighted norm defined by $|v|_B^2  = v^* B^{-1} v $ for any positive-definite $B$.
The corresponding weighted Euclidean inner-product is given by
$\langle \cdot,  \cdot \rangle_{B}:=\langle \cdot, B^{-1}\cdot \rangle.$
We use $\otimes$ to denote the outer product between
two vectors: $(a \otimes b)c=\langle b,c \rangle a.$

\section{Ensemble Kalman State Estimation}
\label{sec:state}

\subsection{Filtering Problem} \label{sec:filtering_problem}
Consider the discrete-time dynamical system with noisy state transitions and noisy observations in the form:
\begin{align*}
	\text{Dynamics Model:}  \quad v_{j+1} &= \Psi(v_j) + \xi_j, \quad j \in \mathbb{Z}^+ \\
	\text{Data Model:}  \quad y_{j+1} &= Hv_{j+1} + \eta_{j+1}, \quad j \in \mathbb{Z}^+ \\
	\text{Probabilistic Structure:}  \quad v_0 &\sim N(m_0, C_0), \quad \xi_j \sim N(0, \Sigma), \quad \eta_j \sim N(0, \Gamma)\\
 \text{Probabilistic Structure:} \quad v_0 &\perp \{\xi_j\} \perp \{\eta_j\} \text{ independent}
\end{align*}

We assume that $\cH_1, \cH_2$ are finite dimensional Hilbert spaces. Then
$v_j \in \cH_1$,  and \(\Psi:\cH_1 \mapsto \cH_1\) is the state-transition operator. The operator \(H:\cH_1 \mapsto \cH_2\) is the linear observation operator
and $y_j \in \cH_2.$
The covariance operators $C_0,\Sigma$ are assumed
to be invertible.
The objective of filtering is to estimate the state $v_j$ of the dynamical
systems at time $j$, given the data $\{y_\ell\}_{\ell=1}^{j}.$

\begin{remark}

\begin{itemize}

\item {We may extend the methodology in this paper to the setting
where $\cH_1, \cH_2$ are separable infinite dimensional Hilbert spaces.
The covariance operators $C_0,\Sigma$ are assumed trace-class on $\cH_1$,
and $\Gamma$ on $\cH_2$ to ensure that the initial condition $v_0$
and the noises $\xi_j$ and $\eta_j$ live in
$\cH_1, \cH_1$ and $\cH_2$ (respectively) with probability one.
The update formulae we derive require operator composition and
inversion, together with minimization of quadratic functionals on
$\cH_1$ subject to convex constraints. Provided all of these operations
can be carried out, then the methods derived here
are well-defined in the general Hilbert space setting.
This fact is important because
it means that the methods derived have a robustness to mesh refinement
and similar procedures arising when the problem of interest is specified
via a partial differential equation, or other infinite dimensional problem.}

\item We restrict attention to linear observation operators $H$ because this leads to solvable quadratic optimization problems within the context of Kalman-based methods.
In principle, a non-linear observation operator could be used, but the optimization problems defining the algorithms arising in this paper might not have a unique solution in this setting.
\end{itemize}
\end{remark}

\subsection{Ensemble Kalman Filter}
The ensemble Kalman filter is a particle-based sequential optimization
approach to the state estimation problem.
The particles are denoted by $\{{v}_{j}^{(n)}\}^{N}_{n=1}$ and represent
a collection of $N$ candidate state estimates at time $j$.
The method proceeds as follows.
The state of all the particles at time $j+1 $ are predicted
using the dynamics model
to give $\{\widehat{v}_{j+1}^{(n)}\}^{N}_{n=1}$.
The resulting empirical covariance of the particles is then used
to define the objective function $I_{{\rm filter},j,n}(v)$, which encapsulates the model-data compromise.
This is minimized in order to obtain the updates $\{{v}_{j+1}^{(n)}\}^{N}_{n=1}.$
{To understand the origin of this optimization perspective on ensemble
Kalman methods we argue as follows. In equation (4.10) of \cite{stuart2015data}, it is shown that the data incorporation step of the Kalman filter may be written as a quadratic optimization problem for the state.
In equation (4.15) of \cite{stuart2015data}, the ensemble Kalman
filter is written by using this quadratic minimization principle with an empirically (from the ensemble) computed covariance.}

The prediction step is
\begin{subequations}
\label{eq:deenz}
\begin{align}
\widehat{v}_{j+1}^{(n)} &= \Psi(v_{j}^{(n)})+\xi^{(n)}_{j}, n=1,...,N \\
\widehat{m}_{j+1} &= \frac{1}{N}\sum^{N}_{n=1} \widehat{v}_{j+1}^{(n)} \\
\widehat{C}_{j+1} &= \frac{1}{N}\sum^{N}_{n=1}\bigl(\widehat{v}^{(n)}_{j+1}-\widehat{m}_{j+1}\bigr)\bigl(\widehat{v}^{(n)}_{j+1}-\widehat{m}_{j+1}\bigr)^{T}. \label{eq2pt1c}
\end{align}
\label{eq:this}
\end{subequations}
Here we have
$\xi_{j}^{(n)} \sim N(0,\Sigma)$ i.i.d..
Because the empirical covariance contains only $N-1$ independent pieces of information, \eqref{eq2pt1c} is sometimes
scaled by $N-1$ and not $N$; making this change would lead to no changes in the statements and proofs of all the theorems, and would only affect the definition of covariance within the algorithms.

Let $\mathcal{R}(\widehat{C}_{j+1})$ denote the range of $\widehat{C}_{j+1}$.
The update step is then
\begin{equation}
v_{j+1}^{(n)}=\underset{v}{\mathrm{argmin}}\,I_{\rm{filter},j,n}(v)
\label{eq:update2}
\end{equation}
where
\begin{align}
\label{eq:deenz2}
I_{{\rm filter},j,n}(v) :=& \left\lbrace
	\begin{array}{ll}
		\frac{1}{2} \vert y_{j+1}^{(n)} - Hv \vert ^{2}_{\Gamma} + \frac{1}{2} \vert v-\widehat{v}_{j+1}^{(n)} \vert ^{2}_{\widehat{C}_{j+1}} & \mbox{if } v-\widehat{v}_{j+1}^{(n)} \in \mathcal{R}(\widehat{C}_{j+1}). \\
		\infty & \mbox{otherwise.}
	\end{array}
	\right.
\end{align}

It can be useful to rewrite the objective function for the
optimization problem in an equivalent
and more standard form for input to software:
\begin{align*}
\left\lbrace
	\begin{array}{ll}
		\frac{1}{2} v^T \Bigr( H^{T} {\Gamma}^{-1} H + \widehat{C}_{j+1}^{-1} \Bigr) v - \Bigr( \widehat{C}_{j+1}^{-1^T} \widehat{v}_{j+1}^{(n)}  + H^T  {\Gamma^{-1}}^T y_{j+1}^{(n)} \Bigr)^T v & \mbox{if } v-\widehat{v}_{j+1}^{(n)} \in \mathcal{R}(\widehat{C}_{j+1}). \\
		\infty & \mbox{otherwise.}
	\end{array}
	\right.
\end{align*}

The $y_{j+1}^{(n)}$ are either identical to the data $y_{j+1}$, or found by
perturbing it randomly.

Note that $\widehat{C}_{j+1}$ is an operator of rank at most $N-1$, and thus can only be invertible when $N-1$ is larger than the dimension of $\cH_1$. For moderate- and high-dimensional systems, it is often impractical to satisfy this condition.
However, the minimizing solution can be found
by regularizing $\widehat{C}_{j+1}$ by addition of $\epsilon I$ for
$\epsilon>0$, deriving the update equations
and then letting $\epsilon \to 0.$ We give the resulting formulae, and then
justify them immediately afterwards, in the following subsubsection.
Alternatively it is possible to directly
seek a solution in $\mathcal{R}(\widehat{C}_{j+1})$, which is a subspace of
dimension $N-1$; this is done in the subsequent subsubsection.

\subsubsection{Formulation In The Original Variables}
{The well-known Kalman update formulae arising from
solution of the minimization problem \eqref{eq:update2}, \eqref{eq:deenz2}}
are as follows:
\begin{subequations}
\label{eq:update}
\begin{align}
S_{j+1} &= H\widehat{C}_{j+1}H^{T} + \Gamma \\
K_{j+1} &=\widehat{C}_{j+1}H^{T}S_{j+1}^{-1} \qquad(\rm{Kalman}\,{\rm Gain}) \\
y_{j+1}^{(n)}&=y_{j+1}+s\eta_{j+1}^{(n)},n=1,...,N\\
v_{j+1}^{(n)} &= (I-K_{j+1}H)\widehat{v}_{j+1}^{(n)}+K_{j+1}y_{j+1}^{(n)},n=1,...,N
\end{align}
\end{subequations}\\
Here
$\eta_{j}^{(n)} \sim N(0,\Gamma)$ i.i.d. and
the constant $s$ takes value $0$ or $1$.
When $s=1$ the $y_{j+1}^{(n)}$ are referred
to as perturbed observations. The choice $s=1$ is made to ensure
the correct statistics of the updates in the linear Gaussian setting when
a probabilistic viewpoint is taken,
and more generally to introduce diversity into the ensemble
procedure when an optimization viewpoint is taken.
Derivation of the formulae may be found in \cite{stuart2015data}.
In brief the formulae arise from completing the square in the
objective function $I_{{\rm filter},j,n}(\cdot)$ and then
applying the Sherman–Morrison formula to rewrite the updates in the
data space rather than state space; the latter is advantageous in many applications where $\cH_2$ has dimension much smaller than $\cH_1.$

We summarize with the following pseudo-code:

\vspace{0.1in}
\begin{algorithm}
\caption{EnKF Algorithm}
\label{alg:enkf}
\begin{algorithmic}[1]
\State Choose $\{v_{0}^{(n)}\}^{N}_{n=1}$, $j=0$
\State Predict $\{\widehat{v}_{j+1}^{(n)}\}^{N}_{n=1}$, $\widehat{C}_{j+1}$ from \eqref{eq:deenz} \label{step1a}
\State Update $\{{v}_{j+1}^{(n)}\}^{N}_{n=1}$ from \eqref{eq:update}
\State $j \gets j+1$, go to~\ref{step1a}.
\end{algorithmic}
\end{algorithm}
\vspace{0.1in}

An equivalent formulation of the minimization problem
is now given by means of a penalized Lagrangian approach
to incorporate the property that the solution of the optimization
problem lies in the range of the empirical covariance. The
perspective is particularly useful when further constraints
are imposed on the solution of the optimization problem.

\begin{theorem}
\label{t:1}
Suppose that the dimensions of $\cH_1$ and $\cH_2$ are finite. Let $j$ be in $\mathbb{Z}^+$ and $1 \leq n \leq N$. Define $y'=y_{j+1}^{(n)}-H\widehat{v}_{j+1}^{(n)}$. {Then the update formulae \eqref{eq:update}, which follow from the
minimization problem \eqref{eq:update2}, \eqref{eq:deenz2},
may be given alternatively as}
\begin{equation}
\label{e:t:1}
v_{j+1}^{(n)}=\widehat{v}_{j+1}^{(n)}+\underset{(a,v') \in \cA}{\mathrm{argmin}}\,\Bigr( \frac{1}{2} \vert y'-Hv' \vert^{2}_{\Gamma} +  \frac{1}{2} \langle a,v' \rangle \Bigr)
\end{equation}
where $\cA=\{(a,v') \in \cH_1 \times \cH_1: \widehat{C}_{j+1}a=v'\}$ and
the argmin is projected from the pair $(a,v')$ onto the $v'$ coordinate only.
Moreover $v_{j+1}^{(n)}=\underset{\epsilon\rightarrow 0}{\mathrm{lim}}\,v_{\epsilon}$ with
\begin{equation}
\label{eq:regmin}
v_{\epsilon}=\underset{v \in \cH_1}{\mathrm{argmin}}\,\Bigl( \frac{1}{2}\vert y_{j+1}^{(n)}-Hv \vert^{2}_{\Gamma} + \frac{1}{2}\vert\widehat{v}_{j+1}^{(n)}-v \vert^{2}_{\widehat{C}_{\epsilon}} \Bigr)
\end{equation}
 and $\hat{C}_{\epsilon}=\widehat{C}_{j+1}+\epsilon I$.
\end{theorem}

\begin{proof}
{For notational convenience denote $\widehat{C}=\hat{C}_{j+1}$.
The objective function $I_{{\rm filter},j,n}(v)$
appearing in \eqref{eq:deenz2} is infinite if and only if
$v-\widehat{v}_{j+1}^{(n)}$ is
in the range of $\widehat{C}$. Thus, since the range of $\widehat{C}$ is non-empty, we may confine minimization to the set of $(a,v') \in \cA.$
Note that $\widehat{C}$ is in general not invertible, as it has rank
{$N-1$}, which may be less than the dimension of $\cH_1$. The set $\cA$ thus
comprises all $v'$ in the range of  $\widehat{C}$ (a convex set)
and, for each such
$v'$ the set of $a$ solving $\widehat{C}a=v'$; such an $a$ is unique
upto translations in the null-space of $\widehat{C}.$ Thus $\cA$ is a convex
set. Notice that, for
such pairs $(a,v') \in \cA$,
$\langle a,v' \rangle = \vert v' \vert^{2}_{\widehat{C}}$
with $v'$ lying in the range of the operator $\widehat{C}$.
Although the element $a$ is uniquely defined only up to translations in the
nullspace of $\widehat{C}$, such translations do not change the value
of the inner product $\langle a,v' \rangle$.}
The restriction of $\widehat{C}$ over the constraint set is positive definite which means that the quadratic objective function, now depending only on $v'$, is strongly convex. Therefore the problem has a unique solution and its Lagrangian is written as:
\begin{equation*}
\mathcal{L}(v',a,\lambda)=\frac{1}{2} |y' - H v'|_{\Gamma}^2 + \frac{1}{2} \langle a, v' \rangle
+\langle \lambda, \widehat{C}a-v' \rangle
\end{equation*}
To express optimality conditions compute the derivatives and set them to zero:
\begin{subequations}
\label{eq:LM_EL}
\begin{align*}
-H^T\Gamma^{-1}(y'-Hv')+\frac12 a-\lambda & =0,\\
\frac12 v'+\widehat{C}\lambda &=0,\\
v'-\widehat{C}a &=0.
\end{align*}
\end{subequations}
The last two equations imply that $\widehat{C}(2\lambda+a)=0.$ Thus we set
$\lambda=-\frac{1}{2}a$ and drop the second equation, replacing the first by
$$-H^T\Gamma^{-1}(y'-H\widehat{C}a)+a=0.$$
Solving the resulting equation for $a$ gives
\begin{align*}
a = (H^T\Gamma^{-1}H\widehat{C}+I)^{-1}H^T\Gamma^{-1}y'.
\end{align*}
From this formula it follows that
\begin{align*}
v_{j+1}^{(n)}&=\hat{v}_{j+1}^{(n)}+v'\\
&=\hat{v}_{j+1}^{(n)}+\widehat{C}a\\
&=\hat{v}_{j+1}^{(n)}+\widehat{C}(H^T\Gamma^{-1}H\widehat{C}+I)^{-1}H^T\Gamma^{-1}y'\\
&=\hat{v}_{j+1}^{(n)}+\widehat{C}(H^T\Gamma^{-1}H\widehat{C}+I)^{-1}H^T\Gamma^{-1}(y_{j+1}^{(n)} - H \hat{v}_{j+1}^{(n)})\,.
\end{align*}
{If we define 
\begin{equation}
\label{eq:K}
K=\widehat{C}(H^T\Gamma^{-1}H\widehat{C}+I)^{-1}H^T\Gamma^{-1}
\end{equation}
then we see that
\begin{equation}
\label{eq:K2}
v_{j+1}^{(n)}=(I-KH)\hat{v}_{j+1}^{(n)}+Ky_{j+1}^{(n)}\,.
\end{equation}
This is precisely the form of the ensemble Kalman update, and 
to complete the proof of the first part of the theorem
it remains to show that this defintion of
$K$ agrees with the formulae given in \eqref{eq:update}; this 
amounts to verifying the identity
\begin{equation}
\label{eq:verify}
(H^T\Gamma^{-1}H\widehat{C}+I)^{-1}H^T\Gamma^{-1}=H^T 
(H\widehat{C}H^T+\Gamma)^{-1}.
\end{equation}
To verify this we start from the matrix identity
$$(H^T\Gamma^{-1}H\widehat{C}+I)^{-1}(H^T\Gamma^{-1}H\widehat{C}H^T+H^T)=H^T$$
noting that it may be factored to write
$$(H^T\Gamma^{-1}H\widehat{C}+I)^{-1}(H^T\Gamma^{-1})(H\widehat{C}H^T+\Gamma)=H^T.$$
Inverting $(H\widehat{C}H^T+\Gamma)$ on the right gives the desired identity
\eqref{eq:verify}.}
\bigskip

{We now study the alternative representation of the minimization
problem \eqref{eq:update2}, \eqref{eq:deenz2}, by \eqref{eq:regmin}. 
We first note that} $H^T\Gamma^{-1}H+\widehat{C}_{\epsilon}^{-1}$ is 
strictly positive definite and hence the related quadratic function is strongly convex. As a consequence we have existence and uniqueness of the solution, and the optimality condition becomes,
\begin{equation*}
(H^T\Gamma^{-1}H+\widehat{C}_{\epsilon}^{-1})v_{\epsilon}=H^T\Gamma^{-1}y_{j+1}^{(n)}+\widehat{C}_{\epsilon}^{-1}\widehat{v}_{j+1}^{(n)}\,.
\end{equation*}
Then if we apply Woodbury matrix identity we obtain
\begin{equation*}
v_{\epsilon}=(\widehat{C}_{\epsilon}-\widehat{C}_{\epsilon}H^T(H\widehat{C}_{\epsilon}H^T+\Gamma)^{-1}H\widehat{C}_{\epsilon})(H^T\Gamma^{-1}y_{j+1}^{(n)}+\widehat{C}_{\epsilon}^{-1}\widehat{v}_{j+1}^{(n)}).
\end{equation*}
{Note that the matrix multiplying $\widehat{v}_{j+1}^{(n)}$ is
$$(\widehat{C}_{\epsilon}-\widehat{C}_{\epsilon}H^T(H\widehat{C}_{\epsilon}H^T+\Gamma)^{-1}H\widehat{C}_{\epsilon})\widehat{C}_{\epsilon}^{-1}=(I-\hat{C}_{\epsilon}H^T(H\hat{C}_{\epsilon}H^T+\Gamma)^{-1}H)$$
and that the  matrix multiplying $y_{j+1}^{(n)}$ is
\begin{align*}
(\widehat{C}_{\epsilon}-&\widehat{C}_{\epsilon}H^T(H\widehat{C}_{\epsilon}H^T+\Gamma)^{-1}H\widehat{C}_{\epsilon})H^T\Gamma^{-1}\\
&=\widehat{C}_{\epsilon}H^T(I-(H\widehat{C}_{\epsilon}H^T+\Gamma)^{-1}H\widehat{C}_{\epsilon}H^T)\Gamma^{-1}\\
&=\widehat{C}_{\epsilon}H^T(H\widehat{C}_{\epsilon}H^T+\Gamma)^{-1}\Gamma\Gamma^{-1}\\
&=\widehat{C}_{\epsilon}H^T(H\widehat{C}_{\epsilon}H^T+\Gamma)^{-1}
\end{align*}
so that}
\begin{equation*}
v_{\epsilon}=(I-\hat{C}_{\epsilon}H^T(H\hat{C}_{\epsilon}H^T+\Gamma)^{-1}H)\hat{v}_{j+1}^{(n)}+\hat{C}_{\epsilon}H^T(H\hat{C}_{\epsilon}H^T+\Gamma)^{-1}y_{j+1}^{(n)}.
\end{equation*}
Finally, as $A \mapsto A^{-1}$ is continuous over the set of invertible matrices, letting $\epsilon \rightarrow 0$ gives:
\begin{equation*}
\underset{\epsilon\rightarrow 0}{\mathrm{lim}}\,v_{\epsilon}=(I-K_{j+1}H)\hat{v}_{j+1}^{(n)}+K_{j+1}y_{j+1}^{(n)}
\end{equation*}
which concludes the proof.
\end{proof}

\subsubsection{Formulation In Range Of The Covariance}

{The minimization problem for each individual particle has a
solution which, when suitably shifted, lies in the range of the empirical covariance. This allows us to seek the solution of the minimization problem as}
a linear combination of a given set of vectors, and to minimize over
the scalars which define this linear combination.
{This reformulation of the optimization problem
is widely employed in a variety of applications, such as weather forecasting, where the number of ensemble members $N$ is much smaller
than the dimension of the data space; this is because the inversion of $S$ to form the Kalman gain $K$ takes place in the data space.}

In order to implement the minimization in the $N$ dimensional subspace
we note that $I_{{\rm filter},j,n}(v)$ is infinite unless
$$v-\widehat{v}_{j+1}^{(n)}=\widehat{C}_{j+1}a$$
for some $a \in {\mathbb {R}^n}.$
From the structure of $\widehat{C}_{j+1}$ given in
(\ref{eq:deenz}c) it follows that
\begin{equation}
\label{eq:deenz3}
v=\widehat{v}_{j+1}^{(n)}+\frac{1}{N}\sum_{m=1}^N b_m e^{(m)},
\quad  e^{(m)}:=\widehat{v}^{(m)}_{j+1}-\widehat{m}_{j+1}.
\end{equation}
Here each unknown parameter $b_m \in \R$ and  $b:=\{b_m\}_{m=1}^N,$ is
the unknown vector to be determined.
This form for $v$ follows from the fact that
\begin{equation}
\label{eq:deenz35}
\widehat{C}_{j+1}=\frac{1}{N}\sum_{m=1}^N e^{(m)} \otimes e^{(m)}
\end{equation}
which in turn implies that
\begin{equation}
\label{eq:deenz4}
\widehat{C}_{j+1} a=\frac{1}{N}\sum_{m=1}^N b_m e^{(m)}.
\end{equation}
Note that the unknown vector $b$ depends on $n$ as we need to solve the
constrained minimization problem for each of the particles, indexed by
$n=1, \dots, N$; we have suppressed the dependence of $b$ on $n$
for notational simplicity.

The expression \eqref{eq:deenz3} for $v$ in terms of the $e^{(m)}$
can be substituted into \eqref{eq:deenz2} to obtain
a functional $J_{{\rm filter},j,n}(b)$ to be minimized over
$b \in {\mathbb {R}^N},$
because $v$ is an affine function of $b.$ Equation \eqref{eq:deenz3}
may be written in compact form as
\begin{equation}
\label{eq:deenzu}
v=\widehat{v}_{j+1}^{(n)}+Bb
\end{equation}
where $B$ is the linear mapping from ${\mathbb {R}^N}$ into $\cH_1$ defined by
$$Bb:=\frac{1}{N}\sum_{m=1}^N b_m e^{(m)}.$$

We now identify $J_{{\rm filter},j,n}(b)$.
We note that \eqref{eq:deenz4} is solved by taking
$$b_m= \langle e^{(m)}, a \rangle.$$
{Although $a$ is not unique, the non-uniqueness stems only
from translations in the nullspace of $\widehat{C}_{j+1}$. Translations
do affect the values taken by the $b_m$, but do not affect the
vector $v$ given by \eqref{eq:deenzu} because they result in changes
to $b$ which are in the null-space of $B$. Furthermore, for any
such solution, independently of which $a$ is chosen,}
$$\frac{1}{2} \vert v-\widehat{v}_{j+1}^{(n)} \vert ^{2}_{\widehat{C}_{j+1}}
= \frac12 \langle a, \widehat{C}_{j+1} a \rangle=\frac{1}{2N}\sum_{m=1}^N b_m^2.$$
Using this and \eqref{eq:deenzu} in the definition of
$I_{{\rm filter},j,n}(v)$ we obtain
$$J_{{\rm filter},j,n}(b)=I_{{\rm filter},j,n}\bigl(\widehat{v}_{j+1}^{(n)}+Bb\bigr)$$
and hence, from \eqref{eq:deenz2},
\begin{subequations}
\label{eq:z2}
\begin{align}
J_{{\rm filter},j,n}(b) :=& \frac{1}{2} \vert y_{j+1}^{(n)} - H\widehat{v}_{j+1}^{(n)} -HB b \vert ^{2}_{\Gamma}+\frac{1}{2N}|b|^2 \\
=&\frac{1}{2}b^T\left( B^TH^T\Gamma^{-1}HB + \frac{1}{N}I\right)b - \left( B^TH^T\Gamma^{-1}(y_{j+1}^{(n)}-H\hat{v}_{j+1}^{(n)})\right)^T b + \text{const.}
\end{align}
\end{subequations}
Once $b$ is determined it may be substituted back into
\eqref{eq:deenzu} to obtain the solution to the minimization problem.

The preceding considerations also yield the following result, concerning
the unconstrained Kalman minimization problem; its proof is a corollary of the
more general Theorem \ref{t:extra} from the next subsection, which includes
constraints in the minimization problem.

\begin{corollary}
\label{c:extra}
Suppose that the dimensions of $\cH_1$ and $\cH_2$ are finite.
Given the prediction (\ref{eq:deenz}a), the unconstrained
Kalman update formulae may be found by minimizing
$J_{{\rm filter},j,n}(b)$ from \eqref{eq:z2}
with respect to $b$ and substituting into \eqref{eq:deenzu}.
\end{corollary}

We summarize the ensemble Kalman state estimation algorithm, using
minimization over the vector $b$, in the following pseudo-code:

\vspace{0.1in}
\begin{algorithm}
\caption{EnKF Algorithm formulated in range of covariance}
\label{alg:enkfRange}
\begin{algorithmic}[1]
\State Choose $\{v_{0}^{(n)}\}^{N}_{n=1}$, $j=0$
\State Predict $\{\widehat{v}_{j+1}^{(n)},e^{(n)}\}^{N}_{n=1}$, from \eqref{eq:deenz} \label{step1b}
\State Optimize $\{{b}^{(n)}\}^{N}_{n=1}$ as argmin of \eqref{eq:z2}
\State Update ${v}_{j+1}^{(n)} = \widehat{v}_{j+1}^{(n)} + Bb^{(n)}$ from \eqref{eq:deenzu}
\State $j \gets j+1$, go to~\ref{step1b}.
\end{algorithmic}
\end{algorithm}
\vspace{0.1in}

\subsection{Constrained Ensemble Kalman Filter}
In this subsection we introduce linear equality and inequality constraints
on the state variable into the ensemble Kalman filter.
We make prediction according to \eqref{eq:deenz}, and then incorporate data
by solving the minimization problem \eqref{eq:deenz2}
subject to the additional constraints
\begin{subequations}
\label{eq:ieqz}
\begin{align}
Fv&=f,\\
Gv &\preceq g.
\end{align}
\end{subequations}
Here $F$ and $G$ are linear mappings which, respectively,
take the state $v$ into the number of equality and inequality constraints;
the notation $\preceq$ denotes inequality componentwise.

\subsubsection{Formulation In The Original Variables}
The preceding considerations lead to the following algorithm for ensemble
Kalman filtering subject to constraints.
{The existence of a solution to the constrained minimization follows from Theorem \ref{t:extra} below.}

\vspace{0.1in}
\begin{algorithm}
\caption{Constrained EnKF Algorithm}
\label{alg:enkfCon}
\begin{algorithmic}[1]
\State Choose $\{v_{0}^{(n)}\}^{N}_{n=1}$, $j=0$
\State Predict $\{\widehat{v}_{j+1}^{(n)}\}^{N}_{n=1}$, $\widehat{C}_{j+1}$ from \eqref{eq:deenz} \label{step1c}
\State Update $\{{v}_{j+1}^{(n)}\}^{N}_{n=1}$ from \eqref{eq:update}
\NoDo
\For {$n=1:N$}
\NoThen
\If {${v}_{j+1}^{(n)}$ violates constraints in \eqref{eq:ieqz}}
	\State ${v}_{j+1}^{(n)} \gets$ argmin of \eqref{eq:deenz2} subject to \eqref{eq:ieqz}
\EndIf
\EndFor
\State $j \gets j+1$, go to~\ref{step1c}.
\end{algorithmic}
\end{algorithm}
\vspace{0.1in}

\subsubsection{Formulation In Range Of The Covariance}
The linear constraints \eqref{eq:ieqz} can be rewritten in terms
of the vector $b$, by means of \eqref{eq:deenzu}, as follows:
\begin{subequations}
\label{eq:ieqzz}
\begin{align}
FBb&=f-F\widehat{v}_{j+1}^{(n)},\\
GBb &\preceq g-G\widehat{v}_{j+1}^{(n)}.
\end{align}
\end{subequations}

We may thus predict and then optimize the objective
function $J_{{\rm filter},j,n}(b)$, given by \eqref{eq:z2},
subject to the constraints \eqref{eq:ieqzz}.
Implementation of this leads to following algorithm for ensemble
Kalman filtering subject to constraints:

\vspace{0.1in}
\begin{algorithm}
\caption{Constrained EnKF Algorithm formulated in range of covariance}
\label{alg:enkfRangeCon}
\begin{algorithmic}[1]
\State Choose $\{v_{0}^{(n)}\}^{N}_{n=1}$, $j=0$
\State Predict $\{\widehat{v}_{j+1}^{(n)},e^{(n)}\}^{N}_{n=1}$, from \eqref{eq:deenz} \label{step1d}
\State Update ${b}^{(n)} \gets$ argmin of \eqref{eq:z2}, ${v}_{j+1}^{(n)} = \widehat{v}_{j+1}^{(n)} + Bb^{(n)}$ from \eqref{eq:deenzu}
\NoDo
\For {$n=1:N$}
\NoThen
\If {${v}_{j+1}^{(n)}$ violates constraints in \eqref{eq:ieqz}}
	\State ${b}^{(n)} \gets$ argmin of \eqref{eq:z2} subject to \eqref{eq:ieqzz}
	\State Update ${v}_{j+1}^{(n)} = \widehat{v}_{j+1}^{(n)} + Bb^{(n)}$ from \eqref{eq:deenzu}
\EndIf
\EndFor

\State $j \gets j+1$, go to~\ref{step1d}.
\end{algorithmic}
\end{algorithm}
\vspace{0.1in}

Justification for the use of this algorithm,
working in the constrained space parameterized by $b$, is a consequence
of the following:

\begin{theorem}
\label{t:extra}
Suppose that the dimensions of $\cH_1$ and $\cH_2$ are finite.
The problem of finding $v^{(n)}_{j+1}$ as the minimizer of
$I_{{\rm filter},j,n}(v)$ subject to
the constraints \eqref{eq:ieqz} is equivalent to finding $b$ to minimize
$J_{{\rm filter},j,n}(b)$ subject to the constraints \eqref{eq:ieqzz} and
then using \eqref{eq:deenzu} to find $v_{j+1}^{(n)}$ from $b$. Furthermore,
both of these constrained minimization problems have a unique solution provided
that the constraint sets are non-empty.
\end{theorem}

\begin{proof}
For notational convenience set $\widehat{v} = \widehat{v}^{(n)}_{j+1}$, $y = y^{(n)}_{j+1}$, $y' = y - H\widehat{v}$, $\widehat{C}= \widehat{C}_{j+1}$ and $\widehat{C}_{\epsilon} = \widehat{C}_{j+1} + \epsilon I$.
\bigskip

Denote
\begin{align}
&	\begin{aligned}
	v^{*} =
	& \underset{v'}{\text{argmin}}
	& & \frac{1}{2} \vert y' - H v' \vert_{\Gamma}^2 + \frac{1}{2} \langle a,v' \rangle \\
	& \text{subject to}
	& & \bullet \widehat{C}a = v'\\
	&&& \bullet Fv' = f - F\widehat{v}\\
	&&& \bullet Gv' \preceq g - G\widehat{v}\\
	\end{aligned}\label{eq1}\\
&	\begin{aligned}
	v_{\epsilon} =
	& \underset{v}{\text{argmin}}
	& & \frac{1}{2} \vert y - H v \vert_{\Gamma}^2 + \frac{1}{2} \vert v - \widehat{v} \vert_{\widehat{C}_{\epsilon}}^2\\
	& \text{subject to}
	& & \bullet Fv = f\\
	&&& \bullet Gv \preceq g\\
	\end{aligned}
	\label{eq2}
\end{align}
and
\begin{align*}
J(v) &=  \frac{1}{2} \vert y - H v \vert_{\Gamma}^2 + \frac{1}{2} \vert v - \widehat{v} \vert_{\widehat{C}}^2\\
J_{\epsilon}(v) &= \frac{1}{2} \vert y - H v \vert_{\Gamma}^2 + \frac{1}{2} \vert v - \widehat{v} \vert_{\widehat{C}_{\epsilon}}^2
\end{align*}
The part of the statement of Theorem \ref{t:extra} concerning existence of a minimizer
is a consequence of the Lemma \ref{l:extra} stated and proved below.
The second part, concerning the equivalence of minimization over $b$ and over
$v$ (or $v'$) was {shown in equations \eqref{eq:deenz3}-\eqref{eq:z2}.}
This concludes the proof.
\end{proof}

{Notice that in the following lemma, the variables $v_{\epsilon}$, $\widehat{v}$, and $v^{*}$ are analogous to $v_{j+1}^{(n)}$, $\widehat{v}_{j+1}^{(n)}$, and the minimizer over $v'$ from \eqref{e:t:1}, respectively.}

\begin{lemma}
\label{l:extra}
Suppose that the constraint sets of \eqref{eq1} and \eqref{eq2} are non empty, then $v^{*}$ exists and is unique and for all $\epsilon > 0$, $v_{\epsilon}$ exists and is unique. Furthermore $\underset{\epsilon \rightarrow 0}{lim \, v_{\epsilon}} = {\widehat{v} + v^{*}}$.
\end{lemma}

\begin{proof}
{The proof is broken into two parts. In the first we prove existence and
uniqueness of a solution by using the idea that the constraint relating
$a$ and $v'$ renders the problem convex. In the second part of the proof
we study the $\epsilon \to 0$ limit of the regularized solution,
extracting convergent
subsequences through compactness, and demonstrating that they converge
to the desired limit.}
To prove existence and uniqueness of the solution of \eqref{eq1}, notice that it can be reformulated as
\begin{equation*}
	\begin{aligned}
	& \underset{v'}{\text{argmin}}
	& & J({\widehat{v} + v'}) \\
	& \text{subject to}
	& & \bullet \widehat{C}a = v'\\
	&&& \bullet Fv' = f - F\widehat{v}\\
	&&& \bullet Gv' \preceq g - G\widehat{v}\\
	\end{aligned}
\end{equation*}
and that the restriction of $\widehat{C}$ over its range is strictly positive definite. Hence $J$ is a strongly convex function being minimized over a non empty closed convex set. From standard theory $v^*$ exists and is unique. Then as $\widehat{C}_{\epsilon}$ is strictly positive definite, the same type of arguments provide existence and uniqueness of $v_{\epsilon}$.
\medskip

Now we prove the second part of the lemma. We note that ${\widehat{v} + v^{*}}$ matches the constraints of \eqref{eq2}. It follows that for all $\epsilon > 0$, $J_{\epsilon}(v_{\epsilon}) \leq J_{\epsilon}({\widehat{v} + v^{*}})$.
Then let us prove that $J_{\epsilon}({\widehat{v} + v^{*}}) \underset{\epsilon \rightarrow 0}{\rightarrow} J({\widehat{v} + v^{*}})$. First denote by $\lambda_{1} \leq \dots \leq \lambda_{N-1}$ the strictly positive eigenvalues of $\widehat{C}$ (recall that $\widehat{C}$ is symmetric positive semidefinite and that $\mathrm{rank}(\widehat{C}) = N-1$ almost surely). Hence $\widehat{C}_{\epsilon}^{-1} = \sum_{k=1}^{N-1} \frac{1}{\lambda_k + \epsilon}a_ka_k^{T} + \sum_{k=N}^{\mathrm{dim}(\mathcal{H}_1)}\frac{1}{\epsilon}a_ka_k^{T}$ where
the $a_k$'s are the eigenvectors of $\widehat{C}$ (the first and second sums respectively gather the vectors of the range and of the nullspace of $\widehat{C}$) . As $v^*$ lies in the range of $\widehat{C}$, it holds that $\vert v^{*} + \widehat{v} - \widehat{v} \vert_{\widehat{C}_{\epsilon}}^2 = \, \vert v^{*} \vert_{\widehat{C}_{\epsilon}}^2 = \sum_{k=1}^{N-1} \frac{1}{\lambda_k + \epsilon}(a_k^{T}v^{*})^2$. Now as the $a_k$'s do not depend on $\epsilon$, by letting $\epsilon$ tending to zero, this quantity will tend to $$\sum_{k=1}^{N-1}\frac{1}{\lambda_k}(a_k^{T}v^{*})^2=\vert v^{*}\vert_{\widehat{C}}^2=\vert v^{*} + \widehat{v} - \widehat{v} \vert_{\widehat{C}}^2.$$
Therefore it holds that $J_{\epsilon}({\widehat{v} + v^{*}}) \underset{\epsilon \rightarrow 0}{\rightarrow} J({\widehat{v} + v^{*}})$. From this we deduce that there exists $\delta > 0$ such that for all $0 < \epsilon < \delta$, $J_{\epsilon}(v_{\epsilon}) \leq J({\widehat{v} + v^{*}})+1$.
\medskip

Then set $w_{\epsilon}=v_{\epsilon}-\widehat{v}=w_{\epsilon}^{0}+w_{\epsilon}^{1}$ where $w_{\epsilon}^{0}$ lies in the nullspace of $\widehat{C}$ and $w_{\epsilon}^{1}$ in its range (recall that for a symmetric matrix nullspace and range are orthogonal) and see that $J_{\epsilon}(v_{\epsilon}) = \frac{1}{2} \vert y' - Hw_{\epsilon}  \vert_{\Gamma}^2 + \frac{1}{2} \vert w_{\epsilon} \vert_{\widehat{C}_{\epsilon}}^2$. It holds that $\frac{1}{2} \vert w_{\epsilon} \vert_{\widehat{C}_{\epsilon}}^2 \leq J_{\epsilon}(v_{\epsilon}) \leq J({\widehat{v} + v^{*}})+1$ for $\epsilon$ sufficiently small. Furthermore $\vert w_{\epsilon} \vert_{\widehat{C}_{\epsilon}}^2 = \vert w_{\epsilon}^{0} \vert_{\widehat{C}_{\epsilon}}^2+\vert w_{\epsilon}^{1} \vert_{\widehat{C}_{\epsilon}}^2=\frac{1}{\epsilon}\vert w_{\epsilon}^{0} \vert^2 +\vert w_{\epsilon}^{1} \vert_{\widehat{C}_{\epsilon}}^2$, and since this quantity is bounded from above we deduce that $w_{\epsilon}^{0} \underset{\epsilon \rightarrow 0}{\rightarrow} 0$ and that $w_{\epsilon}^{1}$ is bounded. Let $(\epsilon_{m})_{m \in \mathbb{N}}$ be a sequence of positive real numbers such that $\epsilon_{m} \underset{m \rightarrow \infty}{\rightarrow} 0$, and from the preceding extract a converging subsequence (denoted $(\epsilon_{m})_{m \in \mathbb{N}}$ for simplicity) such that $(w_{\epsilon_{m}}^{1})_{m \in \mathbb{N}}$ converges to a limit denoted $w^{*}$. As $w_{\epsilon_{m}}^{1}$ lies in $\mathcal{R}(\widehat{C})$, we can use the eigenvalue decomposition of $\widehat{C}$ to show that $\vert w_{\epsilon_{m}}^{1} \vert_{\widehat{C}_{\epsilon_{m}}}^2 \underset{m \rightarrow \infty}{\rightarrow} \vert w^{*} \vert_{\widehat{C}}^2$.
This limiting identity, and the fact that $w_{\epsilon}^{0}$ has limit $0$,
may be used to establish the first equality within
the following chain of equalities and inequalities:
\begin{align*}
J({\widehat{v} + w^{*}}) &= \underset{m \rightarrow \infty}{\rm lim} \, \frac{1}{2} \vert y' - Hw_{\epsilon_{m}}  \vert_{\Gamma}^2 + \frac{1}{2} \vert w_{\epsilon_{m}}^{1} \vert_{\widehat{C}_{\epsilon_{m}}}^2 \leq \underset{m \rightarrow \infty}{\rm lim} J_{\epsilon_{m}}(v_{\epsilon_{m}})\\
& \leq \underset{m \rightarrow \infty}{\rm lim} J_{\epsilon_{m}}({\widehat{v} + v^{*}})=J({\widehat{v} + v^{*}}).
\end{align*}
Now note that $w^{*}$ matches all the constraints of \eqref{eq1}. Indeed $w_{\epsilon_{m}}^{1}$ lies in the range of $\widehat{C}$ which is a closed space, also $v_{\epsilon_{m}}-\widehat{v}=w_{\epsilon_{m}}^{0}+w_{\epsilon_{m}}^{1} \underset{m \rightarrow \infty}{\rightarrow} w^{*}$. It is clear that $v_{\epsilon_{m}}-\widehat{v}$ matches the equality and inequality constraints of \eqref{eq1} for all $m$ and hence passing to the limit we have that $w^{*}$ satisfies the equalities and inequalities.

From the uniqueness of the minimizer of \eqref{eq1} we have that $w^{*}$ is equal to $v^{*}$. In particular this means that $v^{*}$ is the unique cluster point of the original sequence $(w_{\epsilon_{m}}^{1})_{m \in \mathbb{N}}$. Since the original sequence was arbitrarily chosen, we conclude that $\underset{\epsilon \rightarrow 0}{{\rm lim} \, v_{\epsilon}} = {\widehat{v} + v^{*}}$.
\end{proof}

\begin{remark}
Notice that the proof remains true if we take general convex inequalities. We simply need the constrained sets to be closed and convex; however we have
restricted to linear equality and inequality constraints for simplicity and because these arise most often
in practice.
\end{remark}


\section{Ensemble Kalman Inversion}
\label{sec:inverse}

\subsection{Inverse Problem}

In this section we show how a generic inverse problem may be formulated
as a partially observed dynamical system. This enables the machinery
from the preceding section \ref{sec:state} to be used to solve
inverse problems.

We are interested in the inverse problem of finding $u \in \cH_1$ from
$y \in \cH_2$ where
\begin{align*}
y=G(u)+\eta,\; \eta\sim N(0,\Gamma)\,.
\end{align*}
Time does not appear (explicitly) in this equation (although $G$ may
involve solution of a time-dependent differential equation, for example).
In order to use the ideas from
the previous section, we introduce a new variable $w=G(u)$ and rewrite the
equation as
\begin{align*}
w = &  G(u),\\
y = &  w + \eta.
\end{align*}
The key point about writing the equation this way is that the data $y$
is now linearly related to the variable $v=(u,w)^T$ and now we may apply the ideas of the previous section to the model by introducing the following dynamical system, taking $y_{j+1}=y$ as the given data:
\begin{align*}
u_{j+1} = &  u_j,\\
w_{j+1} = &  G(u_{j}),\\
y_{j+1} = &  w_{j+1} + \eta_{j+1}.
\end{align*}
If we introduce the new variables
\begin{subequations}
\label{eq:these}
\begin{align}
v=(u,w)^T, &\quad \Psi(v)=(u, G(u))^T\\
H =  [0, I],&\quad H^{\perp} = [I,0]\,,
\end{align}
\end{subequations}
and write $v_j=(u_j,w_j)^T$, we may write the dynamical system in the form
\begin{subequations}
\label{eq:new}
\begin{align}
v_{j+1} = & {} \Psi(v_j)\\
y_{j+1} = & {} Hv_{j+1} + \eta_{j+1},
\end{align}
\label{eq:artdyn}
\end{subequations}
\noindent which is exactly in the same form as in the previous section. We note that
$$Hv= w, \quad H^{\perp}v=u.$$

\subsection{Ensemble Kalman Inversion}

The prediction step and the Kalman gain are defined as in \eqref{eq:deenz2}, and
the solution of the optimization problem is given by \eqref{eq:update}. We now
simplify these formulae using the specific structure on $\Psi$, $v$, $H$
arising in the inverse problem and given in \eqref{eq:these}; this results in
block form vectors and matrices. First we note that
\begin{align*}
\widehat{C}_{j+1}=
\begin{bmatrix}
C_{j+1}^{uu} & C_{j+1}^{uw}\\
(C_{j+1}^{uw})^T & C_{j+1}^{ww}
\end{bmatrix}
\;\,, \quad
\bar{v}_{j+1}=
\begin{pmatrix}
\bar{u}_{j+1}\\
\bar{w}_{j+1}
\end{pmatrix}\,.
\end{align*}
Here $$\bar{u}_{j+1} = \frac{1}{N}\sum_{n=1}^Nu_{j}^{(n)},\;\bar{w}_{j+1} = \frac{1}{N}\sum_{n=1}^NG(u_{j}^{(n)}):=\bar{G}_j$$
and
\begin{align*}
C_{j+1}^{uw} &= \frac{1}{N}\sum_{n=1}^N(u_{j}^{(n)}-\bar{u}_{j+1})\otimes(G(u_{j}^{(n)})-\bar{G}_{j}),\\
C_{j+1}^{ww} &= \frac{1}{N}\sum_{n=1}^N(G(u_{j}^{(n)})-\bar{G}_{j})\otimes(G(u_{j}^{(n)})-\bar{G}_{j}),\\
C_{j+1}^{uu} &= \frac{1}{N}\sum_{n=1}^N(u_{j}^{(n)}-\bar{u}_{j+1})\otimes(u_{j}^{(n)}-
\bar{u}_{j{+1}}).
\end{align*}
The covariance $C^{ww}_{j+1}$ denotes the empirical covariance of the
ensemble in data space, $C^{uu}_{j+1}$ denotes the empirical
covariance of the ensemble in space of the unknown $u$,
and  $C^{uw}_{j+1}$ denotes the empirical cross-covariance
from data space to the space of the unknown.

Noting that $S_{j+1}=(C_{j+1}^{ww}+\Gamma)^{-1}$
we obtain
\begin{align}
K_{j+1}=
\begin{pmatrix}\label{eq:2}
C_{j+1}^{uw}(C_{j+1}^{ww}+\Gamma)^{-1}\\
C_{j+1}^{ww}(C_{j+1}^{ww}+\Gamma)^{-1}
\end{pmatrix}\,.
\end{align}
Combining equation \eqref{eq:2} with the update equation within
\eqref{eq:update} it follows that
\begin{align*}
\{v_j^{(n)}\}_{n=1}^{N} \rightarrow \{v_{j+1}^{(n)}\}_{n=1}^{N}
\end{align*}
and
\begin{align*}
\{H^{\perp}v_j^{(n)}\}_{n=1}^{N} \rightarrow \{H^{\perp}v_{j+1}^{(n)}\}_{n=1}^{N}
\end{align*}
and hence that
\begin{align*}
u_{j+1}^{(n)} = H^{\perp}v_{j+1}^{(n)} = u_j^{(n)}+ C_{j+1}^{uw}\bigl(C_{j+1}^{ww}+\Gamma\bigr)^{-1}\bigl(y_{j+1}^{(n)}-G(u_j^{(n)})\bigr).
\end{align*}
Thus we have derived the EKI update formula:
\begin{equation}
\label{eq:EnKF}
u_{j+1}^{(n)}  = u_j^{(n)}+ C_{j+1}^{uw}\bigl(C_{j+1}^{ww}+\Gamma\bigr)^{-1}\bigl(y_{j+1}^{(n)}-G(u_j^{(n)})\bigr).
\end{equation}
We note also that
\begin{equation}
\label{eq:EnKF2}
w_{j+1}^{(n)}  = G(u_j^{(n)})+ C_{j+1}^{ww}\bigl(C_{j+1}^{ww}+\Gamma\bigr)^{-1}\bigl(y_{j+1}^{(n)}-G(u_j^{(n)})\bigr).
\end{equation}
However $w_{j+1}^{(n)}$ is not needed to update the state and so plays no role in
this unconstrained EKI algorithm. (It may be used, however, to impose constraints
on observation space, as discussed in the next subsection.)

In summary we have derived the following algorithm for solution of the
unconstrained inverse problem:

\vspace{0.1in}
\begin{algorithm}
\caption{EKI Algorithm}
\label{alg:eki}
\begin{algorithmic}[1]
\State Choose $\{u_{0}^{(n)}\}^{N}_{n=1}$, $j=0$
\State Calculate forward model applications $\{G(u_j^{(n)})\}^{N}_{n=1}$
\State Update $\{{u}_{j+1}^{(n)}\}^{N}_{n=1}$ from \eqref{eq:EnKF}
\State $j \gets j+1$, go to~\ref{step1a}.
\end{algorithmic}
\end{algorithm}
\vspace{0.1in}

\subsection{Ensemble Kalman Inversion With Constraints}

\subsubsection{Formulation In The Original Variables}

We now consider imposing constraints on the optimization step
arising in ensemble Kalman inversion. As in the unconstrained
case we do this by formulating the problem as a special case of
the partially observed dynamical system, subject to constraints,
from the previous section.

To this end we formulate the constraints in the space of the unknown
and the data as follows:
\begin{subequations}
\label{eq:ieqz2}
\begin{align}
F^{u}u &= f^{u},\\
F^{w}w &= f^{w},\\
G^{u}u &\preceq g^{u},\\
G^{w}w &\preceq g^{w}{.}
\end{align}
\end{subequations}
The algorithm proceeds by predicting according to equation \eqref{eq:deenz},
and then optimizing \eqref{eq:deenz2}, all using the specific structure
\eqref{eq:these}, and with the optimization subject to the constraints
\eqref{eq:ieqz2}, written in the notation of the general Kalman
updating formulae in \eqref{eq:ieqz3-1}, detailed below;
in particular the rewrite \eqref{eq:ieqz3-1} of the constraints expresses
everything in terms of the variable $v$.
We may summarize the constraints as follows, to allow
direct application of the ideas of the previous section.
To this end define
\begin{subequations}
\label{eq:missed}
\begin{align}
F &= \begin{pmatrix}
F^{u}H^{\perp}\\
F^{w}H
\end{pmatrix}
= \begin{pmatrix}
F^{u} & 0\\
0 & F^{w}
\end{pmatrix}\\
G &= \begin{pmatrix}
G^{u}H^{\perp}\\
G^{w}H
\end{pmatrix}
= \begin{pmatrix}
G^{u} & 0\\
0 & G^{w}
\end{pmatrix}\\
f &= \begin{pmatrix}
	f^{u}\\
	f^{w}
\end{pmatrix}
,\, g = \begin{pmatrix}
	g^{u}\\
	g^{w}
		\end{pmatrix}\,.
\end{align}
\end{subequations}
Then the constraints \eqref{eq:ieqz2} may be written as
\begin{subequations}
\label{eq:ieqz3-1}
\begin{align}
Fv&=f,\\
Gv &\preceq g.
\end{align}
\end{subequations}
See Algorithm \ref{alg:6} for the resulting pseudo-code.

\vspace{0.1in}
\begin{algorithm}
\caption{Constrained EKI Algorithm}
\label{alg:ekiCon}
\begin{algorithmic}[1]
\State Choose $\{u_{0}^{(n)}\}^{N}_{n=1}$, $j=0$
\State Calculate forward model application $\{G(u_j^{(n)})\}^{N}_{n=1}$
\State Update $\{{u}_{j+1}^{(n)}\}^{N}_{n=1}$ from \eqref{eq:EnKF}
\State Update $\{{w}_{j+1}^{(n)}\}^{N}_{n=1}$ from \eqref{eq:EnKF2}
\NoDo
\For {$n=1:N$}
\NoThen
\If {${v}_{j+1}^{(n)}=({u}_{j+1}^{(n)},{w}_{j+1}^{(n)})$ violates constraints in \eqref{eq:ieqz3-1}}
\State ${v}_{j+1}^{(n)} \gets$ argmin of \eqref{eq:deenz2} subject to
\eqref{eq:these}, \eqref{eq:ieqz3-1}
\EndIf
\EndFor
\State Extract ${u}_{j+1}^{(n)}=H^{\perp} {v}_{j+1}^{(n)}.$
\State $j \gets j+1$, go to~\ref{step1c}.
\end{algorithmic}
\label{alg:6}
\end{algorithm}
\vspace{0.1in}

\subsubsection{Formulation In Range Of The Covariance}

We describe an alternative way to approach the derivation of the EKI update
formulae. We apply Theorem \ref{t:extra} with the specific structure
\eqref{eq:these}, \eqref{eq:ieqz3-1} arising from the dynamical system
used in EKI. To this end we define
\begin{subequations}
\label{eq:z22}
\begin{align}
J_{{\rm filter},j,n}(b) :=& \frac{1}{2} \mid y_{j+1}^{(n)} - G({u}_{j}^{(n)})
- B^wb \mid ^{2}_{\Gamma}+\frac{1}{2N}|b|^2\\
=& \frac{1}{2} b^T \left((B^w)^T\Gamma^{-1}B^w+\frac{1}{N}I\right)b - \left((B^w)^T\Gamma^{-1}(y_{j+1}^{(n)}-G(u_j^{(n)}))\right)^Tb +\text{const.}
\end{align}
\end{subequations}
where $b$ is the vector of $N$ scalar weights $b_m$ and
\begin{subequations}
\label{eq:missed2}
\begin{align}
B^ub &= \frac{1}{N}\sum_{m=1}^N b_m\left(u_j^{(m)}-\bar{u}_{j+1}\right),\\
B^wb &= \frac{1}{N}\sum_{m=1}^N b_m\left(G(u_j^{(m)})-\bar{G}_j\right),\\
Bb &=
\left(
\begin{array}{c}
B^u b\\
B^w b
\end{array}
\right).
\end{align}
\end{subequations}

Once this quadratic form has been minimized with respect to $b$ then the
update formula  \eqref{eq:deenz3} gives
\begin{subequations}
\label{eq:deenz33}
\begin{align}
u_{j+1}^{(n)}&={u}_{j}^{(n)}+\frac{1}{N}\sum_{m=1}^N b_m\bigl({u}^{(m)}_{j}-\bar{u}_{j+1}\bigr),\\
w_{j+1}^{(n)}&=G({u}_{j}^{(n)})+\frac{1}{N}\sum_{m=1}^N b_m\bigl(G({u}^{(m)}_{j})-\bar{G}_{j}\bigr).
\end{align}
\end{subequations}
Note that the vector $\{b_m\}$ depends on the particle label $n$;
as in the previous section, we have suppressed this dependence
for notational convenience.
We may now impose linear equality and inequality constraints on both $u$
and $w=G(u)$ (i.e. in parameter and data spaces) and minimize \eqref{eq:z22}
subject to these constraints. To be more specific if we impose the constraints
\eqref{eq:ieqz3-1} expressed in the variable $b$:
\begin{subequations}
\label{eq:ieqz3}
\begin{align}
FBb&=f-F\widehat{v}_{j+1}^{(n)},\\
GBb &\preceq g-G\widehat{v}_{j+1}^{(n)}.
\end{align}
\end{subequations}
Here $F,G,f$ and $g$ are given by \eqref{eq:missed}, $B$ is defined by
\eqref{eq:missed2} and
\begin{equation*}
\widehat{v}_{j+1}^{(n)} =
\left(
\begin{array}{c}
{u}_{j}^{(n)}\\
G({u}_{j}^{(n)})
\end{array}
\right).
\end{equation*}
See Algorithm \ref{alg:7} for the resulting pseudo-code.

\vspace{0.1in}
\begin{algorithm}
\caption{Constrained EKI algorithm formulated in range of covariance}
\label{alg:EKIConstraintedCovariance}
\begin{algorithmic}[1]
\State Choose $\{u_{0}^{(n)}\}^{N}_{n=1}$, $j=0$
\State Calculate forward model application $\{G(u_j^{(n)})\}^{N}_{n=1}$
\State Update $b^{(n)} \leftarrow$ argmin of (\ref{eq:z22}), $\{{u}_{j+1}^{(n)}\}^{N}_{n=1}$ and $\{{w}_{j+1}^{(n)}\}^{N}_{n=1}$ from (\ref{eq:deenz33})
\NoDo
\For {$n=1:N$}
\NoThen
\If ${v}_{j+1}^{(n)}=({u}_{j+1}^{(n)},{w}_{j+1}^{(n)})$ violates constraints in \eqref{eq:ieqz3}
\State $b^{(n)} \leftarrow$ argmin of (\ref{eq:z22}) subject to \eqref{eq:ieqz3}
\State Update $\{{u}_{j+1}^{(n)}\}$ and $\{{w}_{j+1}^{(n)}\}$ from (\ref{eq:deenz33})
\EndIf
\EndFor
\State $j \gets j+1$, go to~\ref{step1c}.
\end{algorithmic}
\label{alg:7}
\end{algorithm}
\vspace{0.1in}

\begin{remark}
As in the previous section, the result holds true for general convex inequality constraints; the linear case is considered for simplicity of exposition, and
because it is most frequently arising in practice.
\end{remark}

\begin{remark}
The EKI algorithm, with or without constraints,
has the following invariant subspace property:
define $\mathcal{A} = span(u^{(n)}_{0})_{n \in \{ 1,\cdots, N \}}$, then
for all $j$ in $\{0,\dots,J\}$ and for all $n$ in $\{1,\cdots,N\}$,
then the $u^{(n)}_{j}$ defined by the three algorithms in this section
all lie in $\mathcal{A}$.
This is a direct consequence of writing the update formulae in terms
of $b$ and noting \eqref{eq:deenz33}.
\end{remark}

We can now state a result analogous to  Theorem~\ref{t:extra}, and with proof
that is a straightforward corollary of that result, using the specific
structure \eqref{eq:these}:

\begin{theorem}
Suppose that the dimensions of $\cH_1$ and $\cH_2$ are finite. Suppose
also  that the specific structure \eqref{eq:these} is applied.
The problem of finding $u_{j+1}^{(n)}$ from the minimizer of
$I_{{\rm filter},j,n}(v)$, defined in \eqref{eq:deenz2} and subject to
the constraint \eqref{eq:ieqz2}, is equivalent to
finding $b$ that minimizes \eqref{eq:z22}, subject to \eqref{eq:ieqz3}, and
then using \eqref{eq:deenz33} to find $u_{j+1}^{(n)}$ from $b$.
Furthermore,
both of these constrained minimization problems have a unique solution provided
that the constraint sets are non-empty.
\end{theorem}

\section{Numerical Results}
\label{sec:num}

This section contains numerical results which demonstrate the benefits of imposing constraints
on ensemble Kalman methods. Subsection \ref{ssec:state} concerns an application of state estimation (using EnKF)
in biomedicine, using real patient data, whilst subsection \ref{ssec:inverse} concerns {an} application of inversion (using EKI) in seismology and employs
simulated data.
When comparing results from the two experiments, recall that iterations of EKI
correspond to an algorithmic dynamics intended to converge to a single distribution (over ensemble members) on the parameters for which we invert, whereas
iterations of EnKF correspond to the incorporation of new data at
every physical measurement time,
and thus the distribution (over ensemble members) is not necessarily expected
to converge as the iteration progresses. {In both applications, minimizations were performed in \textsc{Matlab} using the default interior-point method (\texttt{fmincon}) through the general quadratic programming function, \texttt{quadprog}.}

\subsection{State Estimation}
\label{ssec:state}
Here we present an application of the constrained EnKF to the tracking and forecasting  of
human blood glucose levels. We use self-monitoring data collected by an individual with Type 2 Diabetes. We use the "P1" data set described by
Albers \emph{et al.} in \cite{albers2017personalized}; this dataset
includes measurements of blood glucose and consumed nutrition, and is publicly available on \texttt{physionet.org}.
For more information on the data, and on an unconstrained data assimilation approach using
the unscented Kalman filter, see \cite{albers2017personalized}.
We model the glucose-insulin system
with the ultradian model proposed by \cite{sturis1991computer}. The primary state variables are the
glucose concentration, $G$, the plasma insulin concentration, $I_{p}$, and the interstitial
insulin concentration, $I_{i}$; these three state variables are augmented
with a three stage delay $(h_1,h_2,h_3)$ which encodes a non-linear delayed hepatic glucose response to plasma insulin levels. The resulting ordinary differential equations have the form:
\begin{subequations}
	\label{eq:ultradian}
	\begin{align}
	   \frac{dI_p}{dt}&=f_1(G)-E(\frac{I_p}{V_p}-\frac{I_i}{V_i})-\frac{I_p}{t_p}\\
	   \frac{dI_i}{dt}&=E(\frac{I_p}{V_p}-\frac{I_i}{V_i})-\frac{I_i}{t_i}\\
	   \frac{dG}{dt}&=f_4(h_3)+m_G(t)-f_2(G)-f_3(I_i)G\\
	   \frac{dh_1}{dt}&=\frac{1}{t_d}(I_p-h_1)\label{UM6-4}\\
	   \frac{dh_2}{dt}&=\frac{1}{t_d}(h_1-h_2)\label{UM6-5}\\
	   \frac{dh_3}{dt}&=\frac{1}{t_d}(h_2-h_3)\label{UM6-6}
	\end{align}
\end{subequations}
Here
$m_G(t)$ represents a known rate of ingested carbohydrates,
$f_1(G)$ represents the rate of glucose-dependent insulin production,
$f_2(G)$ represents insulin-independent glucose utilization,
$f_3(I_i)G$ represents insulin-dependent glucose utilization and
$f_4(h_3)$ represents delayed insulin-dependent hepatic glucose production;
the functional forms of these parameterized processes can be found in the appendix, along with a description of model parameters.

In the EnKF setting, we write $u = [I_p, I_i, G, h_1, h_2, h_3]$, and use \eqref{eq:ultradian} to define $F$ such that
$$\frac{du}{dt}= F(u,t,\theta),$$ where $\theta$ contains model parameters.
We then extend the state vector in order to perform joint parameter estimation: $v = [u, R_g]^T$.

For the purposes of this paper, the function $m_G(t)$ may be viewed as known; it is determined from data describing meals consumed by the patient.
Since insulin ($I_p$ and $I_i$) and delay variables ($h_1$, $h_2$, and $h_3$) are not measured, whilst glucose is measured,
we define the measurement operator to be $H=[0,0,1,0,0,0,0]$.
The discrete time forward model is obtained by integrating the deterministic model in \eqref{eq:ultradian}
between consecutive measurement time-points and applying an identity map to $R_g$. Because these time-points may not be equally spaced, and
because the time-dependent forcings (meals) will differ in different time-intervals, this leads
to a map of the form
$$v_{j+1}=\Psi_j(v_j).$$
This is a slight departure from the methodology outlined in section \ref{sec:state}, where $\Psi$ does
not depend on $j$ (autonomous dynamics) but is a straightforward extension which the reader can easily
provide.

We present EnKF results from a single patient's data when run with and without constraints
(Algorithms~\ref{alg:enkf} and \ref{alg:enkfCon} respectively). We performed joint state-parameter
estimation, augmenting the state with parameter $R_g$ (see Appendix for details of where this
parameter appears) and adding identity-map dynamics for parameter $R_g.$
The following constraints were imposed:

\begin{equation}
  \begin{bmatrix}
    0.01 \\ 0.01 \\ 2000 \\ 0.01 \\ 0.01 \\ 0.01 \\ 0
  \end{bmatrix}
\preceq v \preceq
  \begin{bmatrix}
    10000 \\ 10000 \\ 40000 \\ 10000 \\ 10000 \\ 10000 \\ 1000000
  \end{bmatrix}
\end{equation}

Figure~\ref{fig:enkf-state-dist} compares the overall distribution of updated state means over time when running EnKF with and without these state constraints.
While individual particles in this experiment often violated the constraints, the overall updated means did not. Nevertheless, enforcement of lower-bound constraints shifts up the state distribution slightly. Note that upper bound constraints were never violated in this experiment.

Figure~\ref{fig:enkf-one-iter} shows a two-dimensional state projection of updated particles at a given time step before and after applying the constrained optimization. Note that particles may additionally violate constraints in unplotted dimensions---this explains why one particle whose unconstrained update appears to live within the constraints is in fact differently updated under the constrained optimization. Time step 126 was selected for illustrative purposes, and was the measurement event in which particles most often violated the constraints.

Figure~\ref{fig:enkf-vio} depicts the overall frequency of constraint violations. We observe that the the measured state (blood glucose) never violated a constraint, nor did the inferred parameter $R_g$. However, other model states did often violate constraints, and up to $30\%$ $(4/13)$ of particles simultaneously violated the constraints at a single time-step.

By adding constraints, we ensure that all the simulations which constitute the ensemble method are biologically plausible.

\begin{figure}[htbp]
\begin{center}
\includegraphics[trim = 0.in 0in 0.in 0in,clip,width = \textwidth]{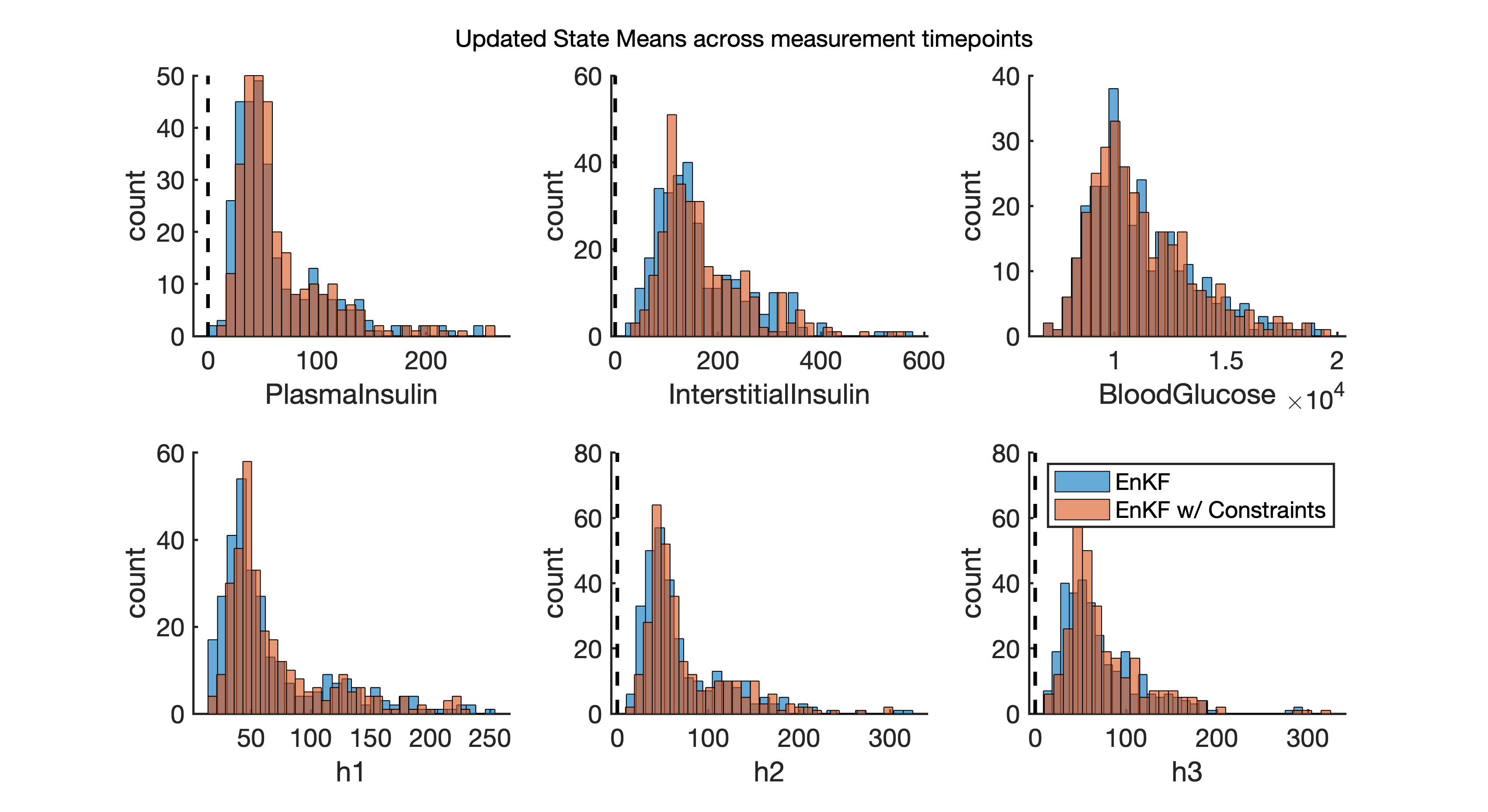}
\end{center}
\caption{The distribution of mean state updates when running EnKF with and without inequality constraints. Black vertical lines denote lower bound state constraints.}
\label{fig:enkf-state-dist}
\end{figure}
\begin{figure}[htbp]
\begin{center}
{\includegraphics[trim = 0in 0in 0in 0in,clip,width= \textwidth]{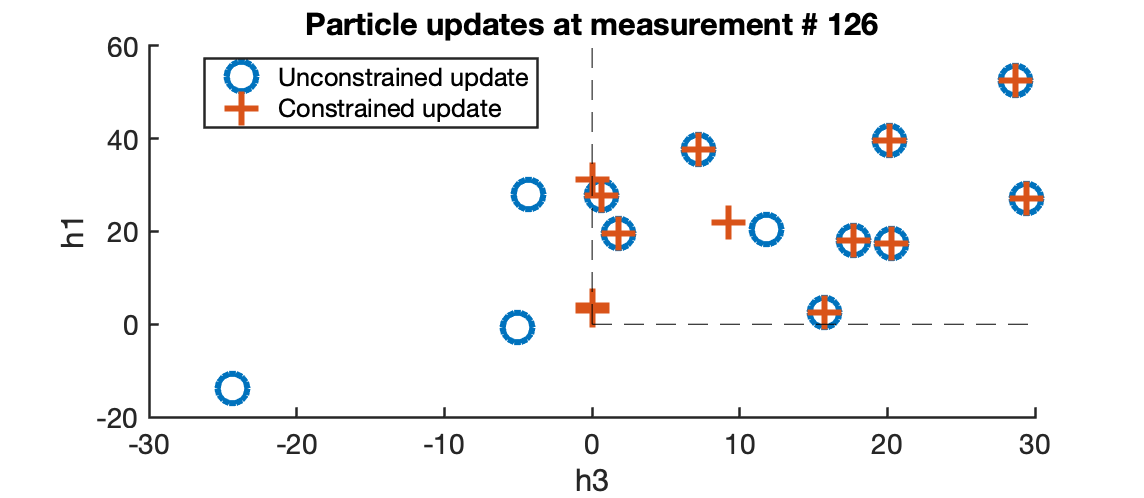}}
\end{center}
\caption{Particle updates at a given time-step (here, measurement 126) are shown using a traditional Kalman gain versus using the constrained optimization. The black lines denote lower bound constraints on the states $h_1$ and $h_3$.}
\label{fig:enkf-one-iter}
\end{figure}
\begin{figure}[htbp]
\begin{center}
{\includegraphics[scale=0.3]{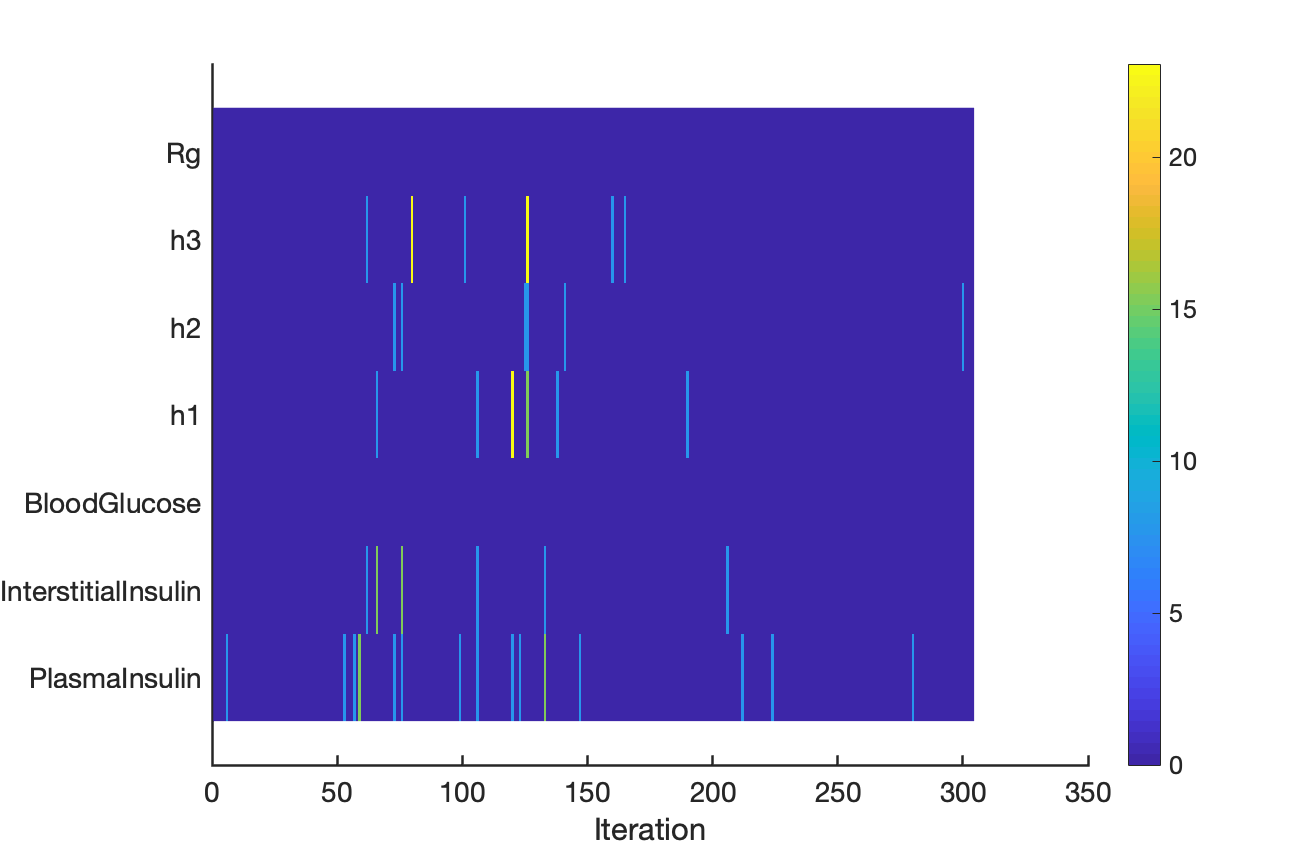}}
\end{center}
\caption{Percentage map of the constraint violations, where each lower-bound constraint is represented by a row. At each iteration, the percentage of particles that violated a constraint is color-coded, with yellow representing the largest proportion of constraint violations.}
\label{fig:enkf-vio}
\end{figure}

\subsection{Inverse Problem}
\label{ssec:inverse}
Here we  present application of the constrained EKI in seismology. We study near-surface site
characterization in which we invert for the shear wave velocity profile of the geomaterials
in the earth shallow crust, using downhole array data. For forward modeling, we consider a
semi-discrete form of the following wave equation in a horizontally stratified heterogeneous soil layer:
\begin{align*}
\frac{\partial}{\partial z} \left[ c_s^2(z) \frac{\partial d(z,t)}{\partial z}\right] - \frac{\partial^2 d(z,t)}{\partial t^2} = 0.
\end{align*}
Here $d(z,t)$ is the displacement field of the wave response as a function of spatial
variable  $z \in (0,H)$ and time variable $t \in (0,T]$. The function $c_s(z)$ is the
shear wave velocity function. We impose the following boundary and initial conditions:
\begin{align*}
d(H,t) = d_0(t), \quad {\partial d(0,t)}/{\partial z} = 0, \quad d(z,0) = 0, \quad {\partial d(z,0)}/{\partial t} = 0
\end{align*}
where $d_0(t)$ is the prescribed displacement at depth $z = H$. Generally, the shear wave velocity changes as a piecewise constant function with depth. If the layering information, i.e., the total number of layers and their thickness, is not available or is poorly characterized, it is desired to use a generic function
for site characterization, such as this:
\begin{align*}
c_s(z) =
\begin{cases}
c_{s0} & 0\leq z\leq z_0\\
c_{s0} \left(1+k(z-z_0)\right)^{n} & z_0\leq z \leq z_1\\
\alpha c_{s0} \left(1+k(z_1-z_0)\right)^{n} & z_1\leq z\leq H
\end{cases}\,.
\end{align*}
See, for example, \cite{shi2018generic}.
In the constrained EKI setting, $u = (c_{s0},k,z_0,n, z_1,\alpha)$ and
$$G(u) = \partial^2 d(0,t)/\partial t^2.$$
For the numerical example studied here, $G^u$ and $g^u$ are determined by enforcing the
constraints $0\leq c_{s0}\leq1000$, $0\leq k \leq 100$, $0 \leq z_0 \leq z_1$, $0\leq n \leq 1$, $ z_0 \leq z_1\leq H$, and $1\leq \alpha \leq 10$. We generate the initial ensemble by drawing samples from uniform distributions and discard members that violate the enforced constraints. In order to avoid very
large velocities at $z=z_1$, we also discard members with $c_s(z_1)>5000$m/s. If we perform
parameter learning using the unconstrained EKI, the experiment fails at $j=1$ because of incapability of the dynamic model to
propagate unphysical values of the shear wave velocity $c_s$.

All results shown use  Algorithm~\ref{alg:EKIConstraintedCovariance}. Figure~\ref{fig:eki-prob} shows
the ensemble distribution of $u$ at $j=2$ before and after enforcing constraints whilst
Figure~\ref{fig:eki-walk} shows the evolution of the updated ensemble. Note
that parameter $k$ saturates with an ensemble close to the upper bound of
$100$ imposed through constraints on this parameter; however experiments
in which we imposed different upper bounds on this parameter lead to
different estimates for $k$, with little change to the estimated
velocity profile and we conclude that this parameter suffers from
identifiability issues.
(Note that Figure~\ref{fig:eki-prob} displays the updated ensemble distribution at a single step in the sequence of ensemble updates, comparing the effect
of imposing constraints with neglecting them; in contrast Figure~\ref{fig:enkf-state-dist} shows the distribution over all measurement time-points of the
ensemble means. The figures thus illustrate different phenomena).

Moreover, Figure~\ref{fig:eki-vio}a shows the map of violation for different constraints enforced on parameters whilst Figure~\ref{fig:eki-vio}b shows the estimated generic $c_s$ profile after 40 iterations compared to the true profile and the initial estimate.
Figure~\ref{fig:eki-vio}a shows the key role employed by the enforcing
of constraints.
In this case the addition of constraints ensures that all the simulations
which constitute the ensemble method are physically meaningful, and also
that the forward model remains well-posed.

\begin{figure}[htbp]
\begin{center}
\includegraphics[trim = 0.in 0in 0.in 0in,clip,width = \textwidth]{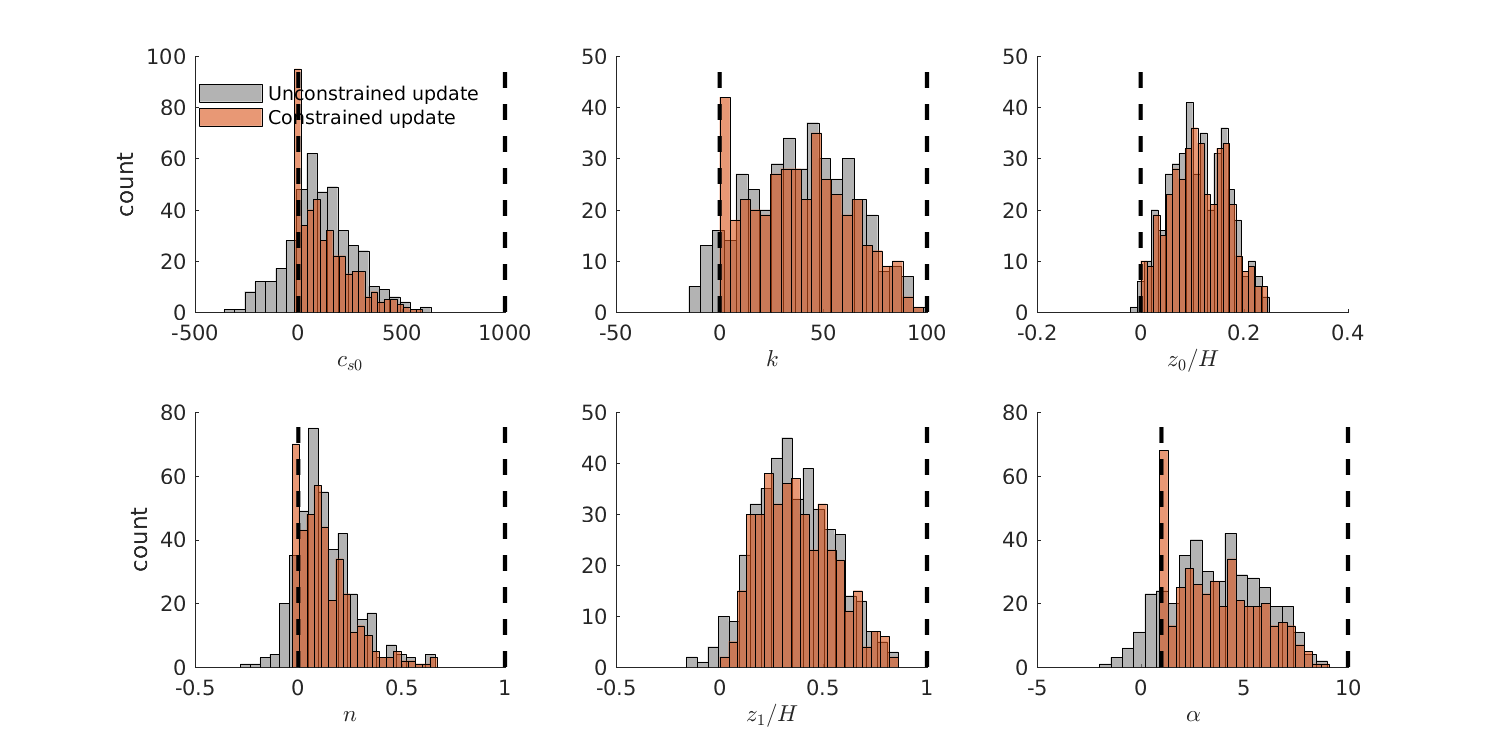}
\end{center}
\caption{The distribution of parameters before and after enforcing constraints in Algorithm~\ref{alg:EKIConstraintedCovariance} at iteration $j=2$. Black vertical lines denote the lower and upper bound constraints.}
\label{fig:eki-prob}
\end{figure}
\begin{figure}[htbp]
\begin{center}
{\includegraphics[trim = 0in 0in 0in 0in,clip,width= \textwidth]{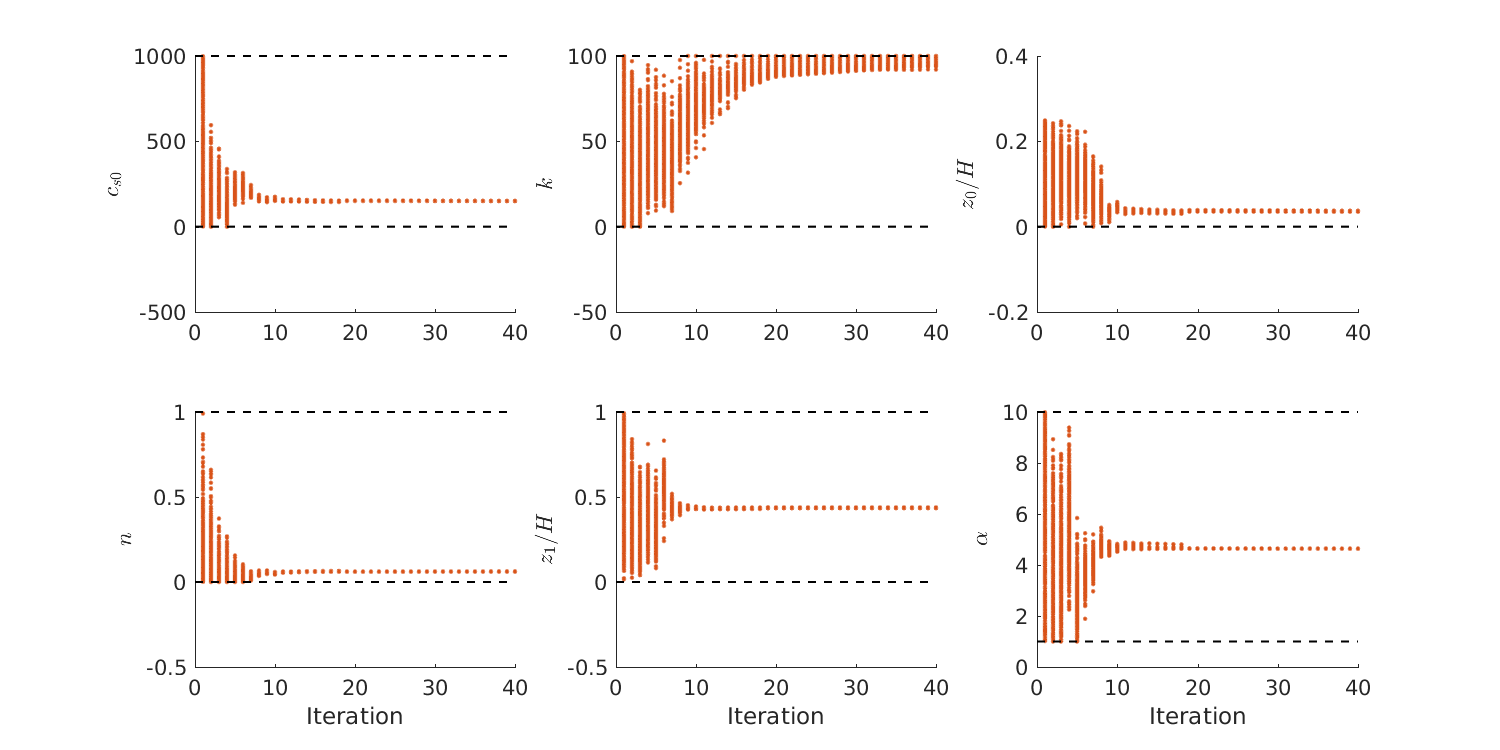}}
\end{center}
\caption{Evolution of the updated ensemble with iteration. Black horizontal lines denote the lower and upper bound constraints.}
\label{fig:eki-walk}
\end{figure}
\begin{figure}[htbp]
\begin{center}
\subfigure[]{\includegraphics[trim = 0in 2.5in 0.9in 2.9in,clip,height = 0.4\textwidth]{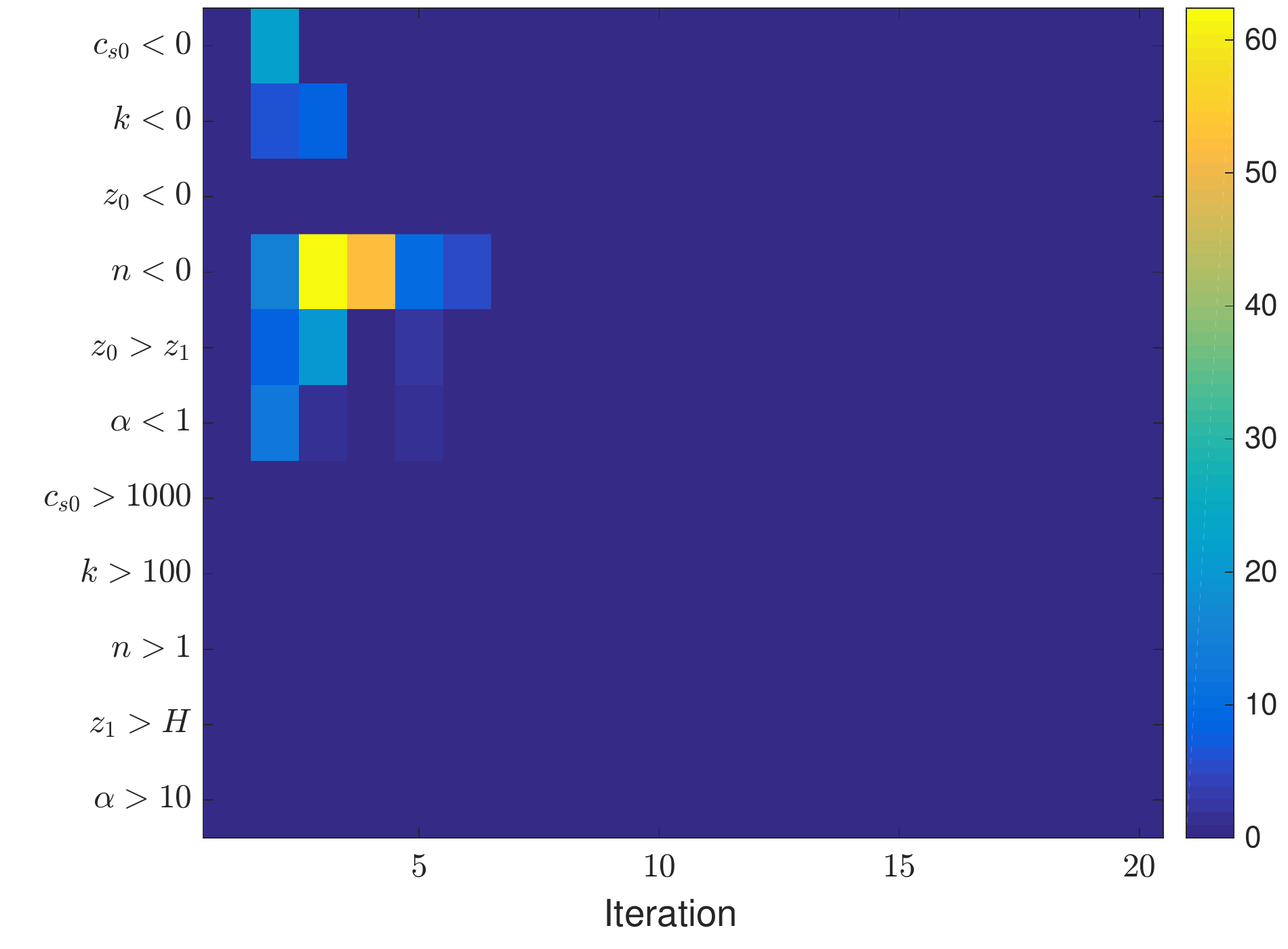}}
\hfill
\subfigure[]{\includegraphics[trim = 0in 0.0in 0.25in 0.15in,clip,height = 0.4\textwidth]{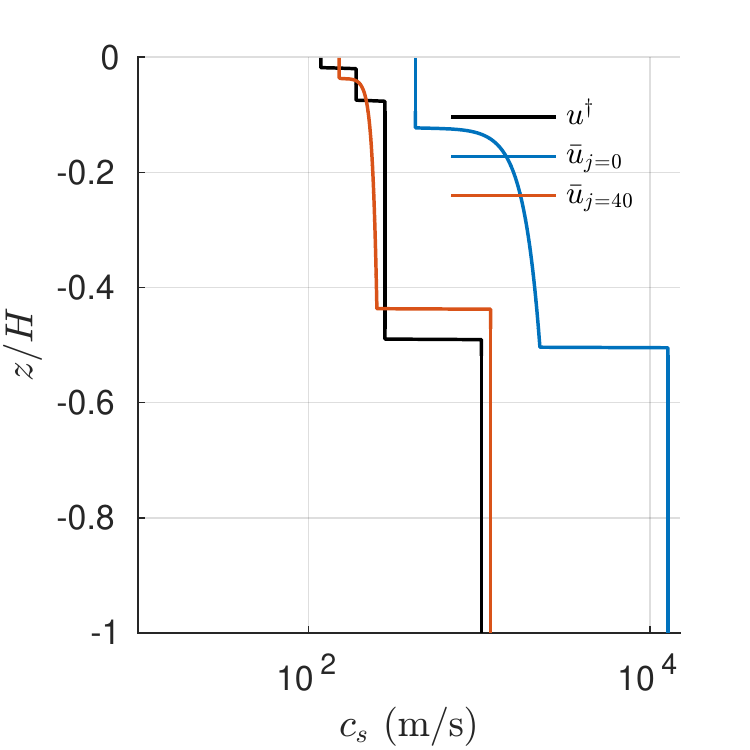}}
\end{center}
\caption{(a) The percentage map of the constraint violations for the first 20 iterations; (b) the estimated velocity profile ($\bar{u}_{j=40}$) compared to the true profile ($u^{\dagger}$) and the initial estimate ($\bar{u}_{j=0}$)}
\label{fig:eki-vio}
\end{figure}

\section{Conclusions}

Constraints arise naturally in many state and parameter estimation problems.
We have shown how convex constraints may be incorporated into ensemble
Kalman based state or parameter estimation algorithms with relatively
few changes to existing code: the standard algorithm is applied and for
any ensemble member which violates a constraint, a quadratic optimization
problem subject to convex constraints is solved instead.
We have written the resulting algorithms in easily digested pseudo-code,
we have developed an underpinning theory and we have given illustrative
numerical examples.

Two primary directions suggest themselves in this area. The first  is the
use of these methods in applications. As indicated in the introduction, our
general formulation is inspired by the two papers
\cite{wang2009state,janjic2014conservation} from the geosciences and we
have demonstrated applicability to problems from biomedicine and
seismology; but many other potential application domains are ripe for
application of ensemble Kalman methodology, because of its black-box
and derivative-free formulation, and the ability to impose constraints
in a straightforward fashion will help to extend this methodology.
The second is the theoretical analysis of these methods: can the inclusion
of constraints be used to deduce improved accuracy of state or parameter
estimates; or can the inclusion of constraints be used to demonstrate improved
performance as measured, for example, by proportion of model runs which are
physically (or biologically etc.) plausible? Furthermore, although the
imposition of constraints is reasonable, it is not clear that it may not
lead to pathologies in algorithmic performance and ruling out, or
understanding, the occurrence of such pathological behaviour may be important.

\ack
This work was funded by NIH-NLM grant RO1 LM012734. AMS was also funded by
AFOSR Grant FA9550-17-1-0185 and by ONR grant N00014-17-1-2079.
\section*{Appendix}

We give the details of the ultradian model of glucose-insulin dynamics
used as the forward model in subsection \ref{ssec:state}. An example
of the induced dynamics is given in Figure \ref{fig:UltradianModel}.

	\begin{align}
	   \frac{dI_p}{dt}&=f_1(G)-E(\frac{I_p}{V_p}-\frac{I_i}{V_i})-\frac{I_p}{t_p}\\
	   \frac{dI_i}{dt}&=E(\frac{I_p}{V_p}-\frac{I_i}{V_i})-\frac{I_i}{t_i}\\
	   \frac{dG}{dt}&=f_4(h_3)+m_G(t)-f_2(G)-f_3(I_i)G\\
	   \frac{dh_1}{dt}&=\frac{1}{t_d}(I_p-h_1)\label{UM6-4}\\
	   \frac{dh_2}{dt}&=\frac{1}{t_d}(h_1-h_2)\label{UM6-5}\\
	   \frac{dh_3}{dt}&=\frac{1}{t_d}(h_2-h_3)\label{UM6-6}
	\end{align}

	where, for $N$ meals at times $\{t_j\}_{j=1}^N$ with carbohydrate composition $\{m_j\}_{j=1}^N$

	\begin{align}
	    m_G(t)=\sum_{j=1}^N \frac{m_j k}{60} \exp(k(t_j-t)), \ \ N=\#\{t_j<t\}\label{UM6-7}
	\end{align}

	and

    \allowdisplaybreaks
	\begin{align}
	    f_1(G)&=\frac{R_m}{1+\exp(\frac{-G}{V_g c_1}+a_1)} :\text{the rate of insulin production}\label{UM6-8}\\
	    f_2(G)&=U_b(1-\exp(\frac{-G}{C_2V_g})) :\text{insulin-independent glucose utilization}\label{UM6-9}\\
	    f_3(I_i)&=\frac{1}{C_3V_g}(U_0+\frac{U_m-U_0}{1+(\kappa I_i)^{-\beta}}), \  f_3(I_i)G:\text{insulin-dependent glucose utilization}\label{UM6-10}\\
	    f_4(h_3)&=\frac{R_g}{1+\exp(\alpha(\frac{h_3}{C_5V_p}-1))} :\text{delayed insulin-dependent glucose utilization}\label{UM6-11}\\
	    \kappa&=\frac{1}{C_4}(\frac{1}{V_i}-\frac{1}{Et_i})
	\end{align}

\begin{figure}[htbp]
\begin{center}
{\includegraphics[scale=0.35]{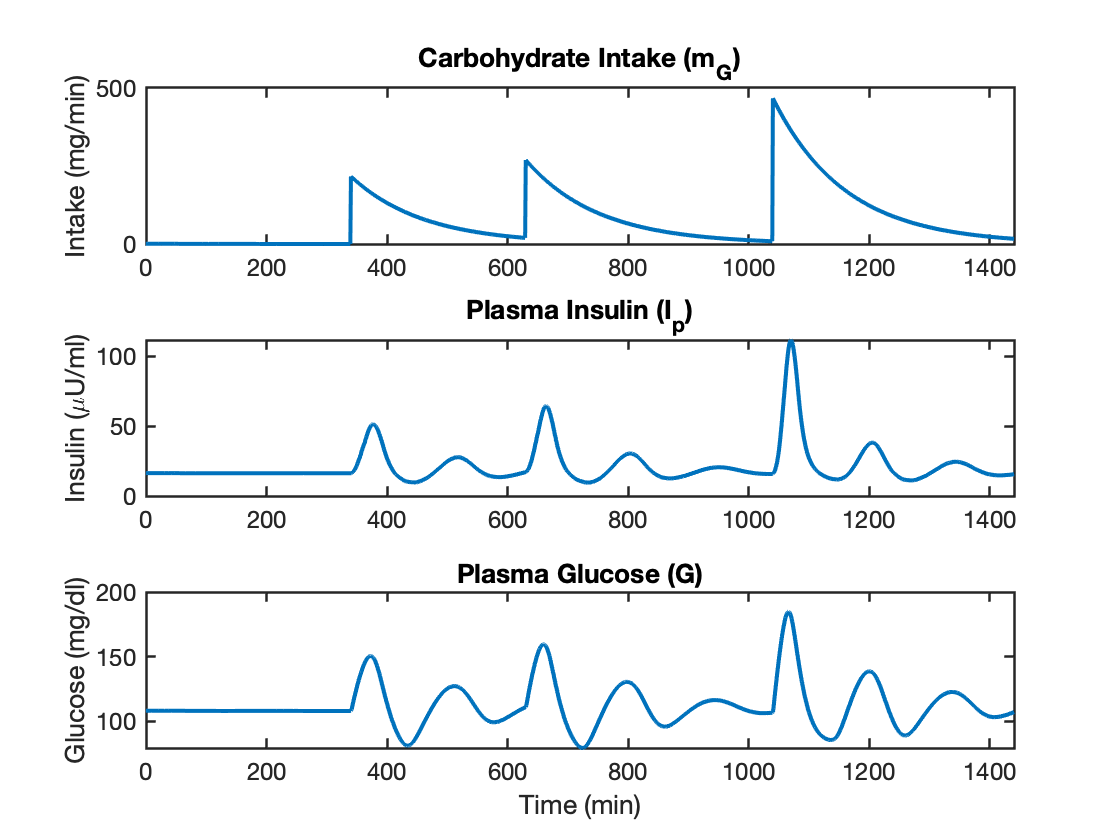}}
\end{center}
    \caption{Here we show the oscillating dynamics of the glucose-insulin response in the ultradian model, driven by an exponentially decaying nutritional driver $m_G$.}
\label{fig:UltradianModel}
\end{figure}

\section*{References}
\bibliography{EnKFmaster}
\bibliographystyle{iopart-num}
\end{document}